\documentclass[12pt]{amsart}

\usepackage{mathpazo}
\usepackage{fullpage}
\usepackage{amssymb, amscd}
\usepackage{enumerate} 
\usepackage{url}
\usepackage{comment}

\title{Equivariant degenerations of spherical modules: part II}

\author{Stavros Argyrios Papadakis}
\address{Department of Mathematics, University of Ioannina, 45110 Ioannina, Greece}
\email{spapadak@cc.uoi.gr}

\author{Bart Van Steirteghem}
\address{Department of Mathematics, Medgar Evers College - City University of New York, 1650 Bedford Ave., Brooklyn, NY 11225, USA}
\email{bartvs@mec.cuny.edu}

\numberwithin{equation}{section}

\newcommand{\inn}{\subseteq}
\newcommand{\isom}{\simeq}

\newcommand{\onto}{\twoheadrightarrow}
\newcommand{\into}{\hookrightarrow}

\renewcommand{\>}{\rangle}
\newcommand{\<}{\langle}

\renewcommand{\epsilon}{\varepsilon}
\newcommand{\eps}{\varepsilon}

\DeclareMathOperator{\Spec}{Spec}
\DeclareMathOperator{\ad}{ad}
\DeclareMathOperator{\Hom}{Hom}

\DeclareMathOperator{\Sym}{Sym}

\newcommand{\mf}{\mathfrak}
\newcommand{\fa}{\mf{a}}

\newcommand{\fg}{\mf{g}}

\newcommand{\ft}{\mf{t}}
\newcommand{\fu}{\mf{u}}

\newcommand{\fz}{\mf{z}}

\newcommand{\gl}{\mf{gl}}

\renewcommand{\sp}{\mf{sp}}

\newcommand{\msf}{\mathsf}
\newcommand{\ssA}{\msf{A}}

\newcommand{\ssD}{\msf{D}}
\newcommand{\ssE}{\msf{E}}

\newcommand{\ssG}{\msf{G}}

\newcommand{\GL}{\mathrm{GL}}
\newcommand{\SL}{\mathrm{SL}}
\newcommand{\SO}{\mathrm{SO}}
\newcommand{\Spin}{\mathrm{Spin}}

\newcommand{\shN}{\mathcal{N}}

\renewcommand{\AA}{{\mathbb A}}
\newcommand{\CC}{{\mathbb C}}
\newcommand{\C}{{\mathbb C}}
\newcommand{\NN}{{\mathbb N}}

\newcommand{\RR}{{\mathbb R}}

\newcommand{\ZZ}{{\mathbb Z}}
\newcommand{\Z}{{\mathbb Z}}

\newcommand{\om}{\omega}

\newcommand{\tHilb}{\mathrm{Hilb}}
\newcommand{\tM}{\mathrm{M}}

\newcommand{\wm}{\mathcal{S}}

\newcommand{\Tad}{T_{\ad}}
\newcommand{\Vgg}{(V/\fg\cdot x_0)^{G_{x_0}}}

\newcommand{\Vggp}{(V/\fg\cdot x_0)^{G'_{x_0}}}
\newcommand{\Vpggp}{(V'/\fg'\cdot x'_0)^{G'_{x_0}}}
\newcommand{\Vg}{V/\fg \cdot x_0}

\newcommand{\sr}{\Pi}
\newcommand{\pr}{R^+}
\newcommand{\wl}{\Lambda}
\newcommand{\dw}{\Lambda^+}
\newcommand{\ms}{\tM_{\wm}}
\newcommand{\tg}{T_{X_0}}

\theoremstyle{plain}
\newtheorem{theorem}{Theorem}[section]
\newtheorem{lemma}[theorem]{Lemma}
\newtheorem{prop}[theorem]{Proposition}
\newtheorem{cor}[theorem]{Corollary}

\theoremstyle{definition}
\newtheorem{defn}[theorem]{Definition}
\newtheorem{remark}[theorem]{Remark}

\newtheorem{example}[theorem]{Example}

\newtheorem{listabc}[theorem]{List}

\newcommand{\loccit}{\emph{loc.cit.}}

\newcommand{\Vbeta}{V^{\beta}}
\newcommand{\VE}{\mathcal{E}}
\newcommand{\Sp}{\mathrm{Sp}}
\newcommand{\Aset}{\mathcal{A}}

\newcommand{\Vggpr}{(V'/\fg\cdot x_0)^{G_{x_0}}}

\newcommand{\rhoprime}{\rho'}

\begin{document}

\begin{abstract}
We determine, under a certain assumption, the Alexeev--Brion moduli scheme $\ms$
of affine spherical $G$-varieties with a prescribed weight monoid $\wm$.
In \cite{degensphermod}  we showed that if $G$ is a
connected  complex reductive group  of type $\ssA$ and
$\wm$ is the weight monoid of
a spherical $G$-module, then $\ms$ is an affine space. Here we prove
that this remains true without any restriction on the type of $G$.
\end{abstract}

\maketitle

\section{Introduction and statement of results} \label{sec:intro}

A natural invariant of an affine variety $X$ equipped with an action of a
complex connected reductive group $G$ is its \textbf{weight monoid}
$\wm(X)$. It is the set of (isomorphism classes of) irreducible
representations of $G$ that occur in the ring of regular functions
$\C[X]$. If every irreducible representation occurs at most once in
this ring,
then $X$ is
called \textbf{multiplicity-free}. If, in addition, $X$ is normal, then it is an \textbf{affine spherical variety}. For multiplicity-free varieties, the weight monoid completely describes the
structure of $\C[X]$ as a representation of $G$. Knop's Conjecture, proved by Losev in \cite{losev-knopconj}, asserts that if $X$ is smooth and multiplicity-free, then $\wm(X)$ uniquely determines $X$. This is no longer true without the smoothness assumption.  A moduli scheme introduced by V.~Alexeev and M.~Brion
\cite{alexeev&brion-modaff} brings geometry to the natural question, ``to what extent does
$\wm(X)$ determine $X$ as a variety?''

To describe the moduli scheme, following \cite[Section
4.3]{brion-ihs}, we introduce
some more notation. Let $B$ be a Borel subgroup of $G$. Then $B=TU$
where $T$ is a maximal torus of $G$ and $U$ is the unipotent radical
of $B$. Let $\dw$ be the monoid of dominant weights in the weight
lattice $\wl$. Recall that by highest weight theory, the elements of $\dw$ are in bijection with
the isomorphism classes of irreducible representations of $G$. Under this identification, the weight monoid $\wm(X)$ of a multiplicity-free affine $G$-variety $X$ is a finitely generated submonoid of $\dw$. Now, given a finitely generated
submonoid $\wm$ of $\dw$, define the following $G$-module: 
\(V(\wm) := \oplus_{\lambda \in \wm} V(\lambda).\) 
By identifying it with the semigroup algebra $\C[\wm]$, we equip the space of highest weight vectors $V(\wm)^U$ with a
$T$-multiplication law. The moduli scheme $\ms$ introduced in \cite{alexeev&brion-modaff}
parametrizes the $G$-multiplication laws on $V(\wm)$ that extend the
chosen $T$-multiplication law on the subspace $V(\wm)^U$. We will sometimes write $\ms^G$ instead of $\ms$ when we need to specify the group under consideration.  Alexeev and
Brion showed that $\ms$ is an affine scheme of finite type over $\C$. 

In more geometric language, the moduli scheme $\ms$ parametrizes
pairs $(X, \varphi)$ where $X$ is a multiplicity-free affine $G$-variety with weight
monoid $\wm$ and $\varphi$ is a $T$-equivariant map $\Spec(\C[X]^U)
\to \Spec(\C[\wm])$.  Alexeev and Brion equipped $\ms$ with a natural action of
the `adjoint torus' $\Tad:=T/Z(G)$, where $Z(G)$ is the center of $G$. They proved that the orbits correspond to isomorphism classes of
multiplicity-free affine varieties with weight monoid $\wm$, and that there
is a unique closed $\Tad$-orbit, which is a fixed point denoted
$X_0$. Finally, they showed that if $X$ is an affine multiplicity-free
variety with weight monoid $\wm$, and we think of $X$ as a closed
point on $\ms$, then the closure of the orbit
$\Tad\cdot X \inn \ms$ has coordinate ring $\C[\Sigma_X]$, where $\Sigma_X$ is
the so-called \textbf{root monoid} of $X$:
\[\Sigma_X :=\<\lambda + \mu - \nu \colon \lambda,\mu,\nu \in \dw
\text{ such that } \<\C[X]_{(\lambda)} \cdot \C[X]_{(\mu)}\>_{\C} \cap
  \C[X]_{(\nu)} \neq 0\>_{\NN}.\]
Here $\C[X]_{(\lambda)}$ is the isotypic component of $\C[X]$ of type
$\lambda \in \dw$. 

\subsection{Main results}
In \cite{degensphermod} we proved that if $\wm$ is the weight monoid of a spherical $G$-module,
where $G$ is of type $\ssA$, then $\ms$ is an affine space. Here we
extend this result to weight monoids
of spherical $G$-modules for  arbitrary connected reductive groups $G$.  That is, we here prove the following.

\begin{theorem} \label{thm:main} 
Assume $W$ is a spherical $G$-module, where $G$ is a
connected reductive algebraic group.  Let
$\wm$ be the weight monoid of  $W$ and let $d_W$ be the rank of the
group $\Z\Sigma_W$ generated by the root monoid $\Sigma_W$ of $W$. Then
\begin{enumerate}[(a)]
\item $\Sigma_W$ is a freely generated monoid; and
\item the $\Tad$-scheme $\tM_{\wm}$ 
is $\Tad$-equivariantly isomorphic to the $\Tad$-module with weight monoid $\Sigma_W$. In particular, the scheme $\tM_{\wm}$ is
isomorphic to  the affine space $\AA^{d_W}$, hence it is irreducible and smooth. 
\end{enumerate} 
\end{theorem}
We recall from \cite[Lemma 2.7]{degensphermod} that for a given spherical $G$-module $W$, the invariant $d_W$ is easy to calculate from the rank of the free abelian group $\<\wm(W)\>_{\Z}$: it is the difference between the rank of $\<\wm(W)\>_{\Z}$ and the number of irreducible components of $W$. 
Thanks to the reduction in Section 4 of \cite{degensphermod}, which is
independent of the type of the group $G$, the proof of
Theorem~\ref{thm:main}, which is formally given in  Section~\ref{sec:formal-proof-theorem}, reduces to the following theorem.

\begin{theorem}\label{thm:cbc}
Suppose $(\overline{G},W)$ is an entry in Knop's  List of saturated indecomposable spherical modules (see List~\ref{KnopL} on page~\pageref{KnopL}).
If $G$ is a connected reductive group such that
\begin{enumerate}
\item $\overline{G}' \inn G \inn \overline{G}$; and
\item $W$ is spherical as a $G$-module
\end{enumerate} 
then $$\dim T_{X_0}\tM^{G}_{\wm} = d_W,$$ where $\wm$ is the weight monoid of $(G,W)$.  
\end{theorem}
In \cite[Section 5]{degensphermod}, we proved Theorem~\ref{thm:cbc}
for groups $\overline{G}$ of type $\ssA$. In Section~\ref{sec:cases}
below, we prove it for the remaining modules in Knop's List,
i.e. those where the acting group contains a component that is not of
type $\ssA$. As in our previous paper, we do this by determining for each entry in Knop's List the structure of $\tg\ms$ as a $\Tad$-module: we determine the $\Tad$-weights occuring in $\tg\ms$ and show that each weight has multiplicity one. It follows from our descriptions that only certain `special' elements of the root lattice of $G$ occur as $\Tad$-weights in $\tg\ms$: every $\Tad$-weight in $\tg\ms$ is a so-called ``spherical root'' of $G$ (cf.\ \cite[Section 1.2]{luna-typeA} for the definition of this notion).   

Section~\ref{sec:crit-extens}, which may be of independent interest,
contains some auxilary results about the tangent space $T_{X_0}\ms$ to
$\ms$ at the point $X_0$.  Corollary~\ref{cor:extending-sections} is a sharpening of the
extension criterion \cite[Proposition 3.4]{degensphermod} for
invariant sections of the normal sheaf of $X_0$ in $V$. 

The Appendix presents the details, in a specific case,  of a 
different technique which explicitly computes the $\Tad$-eigenvectors in $\Vgg$.

After most of the work on this paper had been completed, the preprints  \cite{avdeev&cupit-irrcomps-arxivv2} and \cite{msfwm} were posted on the arXiv. 
We do not use the results contained in these papers, and for the 
weight monoids $\wm$ under consideration in the present paper
our main result is stronger. More precisely, while \cite{avdeev&cupit-irrcomps-arxivv2,msfwm} also prove (and in much greater generality than in the present paper) that the $\Tad$-weights in $\tg\ms$ are spherical roots of $G$ and have multiplicity one, in the present paper we additionally prove that $\ms$ is irreducible for the monoids $\wm$ under consideration. 

\subsection{Formal proof of Theorem~\ref{thm:main}.} \label{sec:formal-proof-theorem}
We now give the proof of Theorem~\ref{thm:main}. Corollary 2.6 and
Corollary 4.17 of \cite{degensphermod} reduce the proof to
Theorem~\ref{thm:cbc}, which we prove by a case-by-case verification
in Section~\ref{sec:cases}.

\subsection{Notations}
We will follow the conventions and notations of
\cite{degensphermod}. In particular, by a  variety we mean a reduced,
irreducible and separated scheme of finite type over $\C$. We will use $\wl$ for the weight
lattice of $G$, i.e. the group of characters of a fixed maximal torus
$T$, which is identified with the group of characters of a chosen
Borel subgroup $B$ of $G$ which contains $T$. Then $\dw$ will denote
the monoid of dominant weights in $\wl$, and we will use $V(\lambda)$
for the irreducible representation of $G$ corresponding to $\lambda
\in \dw$, and $v_{\lambda}$ for a highest weight vector in $V(\lambda)$. We will use $\fg$ for the Lie algebra of $G$.  If $\alpha$ is a root, then $\alpha^{\vee} \in \Hom_{\Z}(\wl,\Z)$ will be
its coroot (in the sense of \cite{bourbaki-geadl47}), $\fg^{\alpha}$ its root space and $X_{\alpha} \in \fg^{\alpha} \setminus \{0\}$ a root operator. We will
use $\sr$ for the set of simple roots (relative to $T$ and $B$) and $\wl_R$ for the root lattice: $\wl_R = \<\sr\>_{\ZZ} \inn \wl$.

We will number the fundamental weights and the simple roots of the simple Lie algebras as in \cite{bourbaki-geadl47}. When $G = \GL(n)$ and $i \in \{1,\ldots, n\}$, the highest weight of the module $\bigwedge^i \CC^n$ will  be denoted by $\omega_i$. Moreover, we put $\omega_0= 0$. It is well-known that the simple roots of $\GL(n)$ have the following expressions in terms of the $\omega_i$:
\begin{equation} 
 \alpha_i = -\omega_{i-1} + 2\omega_i - \omega_{i+1} \quad \text{ for } i\in\{1,2,\ldots,n-1\}.
\end{equation} 

We will use $E^*$ for the basis of a free
monoid $\wm$ of dominant weights and $E:= \{\lambda^* \colon \lambda\in
E^*\}$. Here $\lambda^*$ is the highest weight of the
representation $V(\lambda)^*$ which is dual to $V(\lambda)$; that is:
$V(\lambda^*) \isom V(\lambda)^*$.

\subsection{Acknowledgment}
The authors thank Michel Brion for suggesting, during a 2011 visit of the second-named author to the University of Grenoble, the general strategy for the extension criterion in Section~\ref{sec:crit-extens}. They also thank an anonymous referee for pointing out a mistake in the proof of Proposition~\ref{prop:case10}, for her/his very careful reading and for the numerous and detailed suggestions which improved the paper.

S.~P.\ benefited from experiments with
the computer algebra program Macaulay2~\cite{M2}.
For different parts of the project he was financially supported by
the Portuguese Funda\c{c}\~ao para a Ci\^encia e a Tecno\-lo\-gia through grant
SFRH/BPD/22846/2005 of POCI2010/FEDER and by RIMS, Kyoto University, Japan.

B.~V.~S.\ received support from The City University of New York PSC-CUNY Research Award Program and from the National Science Foundation through grant number DMS-1407394. He also thanks Michel Brion and the Institut Fourier for their hospitality in the Summer of 2011.

\section{Criterion for extension of sections} \label{sec:crit-extens}
In this section, $E^*$ is a set of linearly independent dominant weights of a complex connected reductive group $G$, and $\wm$ is the submonoid of $\dw$ generated by $E^*$. 
We do not assume that $\wm$ is the weight monoid of a spherical module. Like before,
$E=\{\lambda\colon \lambda^* \in E^*\}$. As in \cite{degensphermod}, we put
\begin{align*}
V&:=\oplus_{\lambda \in E} V(\lambda);\\
x_0&:= \sum_{\lambda \in
  E} v_{\lambda} \in V; \\
  X_0&:= \overline{G\cdot x_0}\inn V
  \end{align*}
and we denote by $\shN_{X_0|V}$ the normal sheaf of $X_0$ in $V$. 

\begin{remark} \label{rem:propsX0}
We record some well-known facts about $x_0$ and $X_0$ that will be of use later in the paper.
\begin{enumerate}[(a)]
\item Since $E$ is linearly independent, $\ft \cdot x_0 = \<v_{\lambda} \colon \lambda \in E\>_{\CC}$. \label{item:t0x0} 
\item $X_0$ is a spherical $G$-variety with weight monoid $\wm$; cf.~\cite[Theorem 6]{vin&pop-quasi}
\item By \cite[Theorem 8]{vin&pop-quasi}, the following map is a one-to-one correspondence between the set of subsets of $E$ and the set of $G$-orbits in $X_0$:
\[(D\inn E) \mapsto G\cdot v_D \quad \text{where }v_D:=\sum_{\lambda \in D} v_{\lambda}.\]
\label{item:X0orbits} 
\end{enumerate}
\end{remark}

In \cite{alexeev&brion-modaff} Alexeev and Brion equipped $\ms$ with an action of $\Tad$ and showed that \emph{$X_0$, viewed as a point of $\ms$, is a fixed point and the unique closed orbit for this action.} As in \cite{degensphermod} we will work with a `twist' of the action in \cite{alexeev&brion-modaff}. It is obtained by composing Alexeev and Brion's action with the automorphism of $\Tad$ induced by the automorphism $\gamma \mapsto w_0(\gamma)$ of the root lattice $\Lambda_R$, which is the group of characters of $\Tad$. Here $w_0$ is the longest element of the Weyl group of $(G,T)$. We will call our action on $\ms$ and its induced action on $\tg\ms$ ``the $\Tad$-action.''
As shown in \cite{alexeev&brion-modaff} and reviewed in \cite[\S 2.2]{degensphermod}, we have a sequence of $\Tad$-equivariant linear maps
\begin{equation} \label{eq:14}
\tg\ms \stackrel{\isom}{\longrightarrow} H^0(X_0, \shN_{X_0|V})^G \into H^0(G \cdot x_0,
\shN_{X_0|V})^G \stackrel{\isom}{\longrightarrow} \Vgg,
\end{equation} 
where the first and the third map are isomorphisms, and the second one is an inclusion.  

Because they will play a role later on, we recall from \cite[\S 2.2]{degensphermod} explicit descriptions of our $\Tad$-actions on $\Vgg$ and on $H^0(X_0,\shN_{X_0|V})^G$.  For the former, we begin by equipping $V$ with the same action $\alpha$ of $\Tad$ as in \cite[Definition 2.11]{degensphermod}:
if $t \in T, \lambda \in E$ and $v \in V(\lambda) \inn V$ then
\begin{equation}
\alpha(t,v) = \lambda(t)t^{-1}v. \label{eq:alphadef}
\end{equation}
It follows from highest weight theory that the center $Z(G)$ of $G$ belongs to the kernel of $\alpha$, and therefore $\alpha$ induces an action of $\Tad=T/Z(G)$ on $V$. Let $G \rtimes \Tad$ be the semidirect product of $G$ and $\Tad$, where $\Tad$ acts on $G$ by conjugation. As explained in \cite[p.~102]{alexeev&brion-modaff} the $\Tad$-action $\alpha$ and the linear $G$-action on $V$ can be extended together to a linear action of $G\rtimes \Tad$ on $V$. Since the $\Tad$-action fixes $x_0$, this yields an action of $G_{x_0} \rtimes \Tad$ on $\Vg$, see e.g. \cite[p.~1780]{degensphermod}. It follows that the subspace $\Vgg$ of $\Vg$ is preserved by the action of $\Tad$.  This induced action on $\Vgg$ is what we call ``the $\Tad$-action'' on $\Vgg$. By slight abuse of notation, we also denote it by $\alpha$. 

To describe the $\Tad$-action on $H^0(G\cdot x_0,\shN_{X_0|V})^G$ and on $H^0(X_0,\shN_{X_0|V})^G$, let $\GL(V)^G$ be the group of linear automorphisms of $V$ that commute with the action of $G$. Since the elements of $E$ are distinct, $\GL(V)^G$ is isomorphic to the product of $|E|$ copies of $\C^{\times}$. The natural action of $\GL(V)^G$ on $V$ stabilizes $G\cdot x_0$ and $X_0$ and the embedding $H^0(X_0, \shN_{X_0|V})^G \into H^0(G\cdot x_0, \shN_{X_0|V})^G$ is $\GL(V)^G$-equivariant for the induced actions. Composing the action of $\GL(V)^G$ with the homomorphism
\begin{equation} \label{eq:homomf}
f \colon T \to \GL(V)^G, t \mapsto (\lambda(t))_{\lambda \in E}
\end{equation}
yields an action of $T$ on $V$. We denote the induced $T$-action on $H^0(G\cdot x_0, \shN_{X_0|V})^G$ and on $H^0(X_0, \shN_{X_0|V})^G$ by $\widehat{\psi}$. Proposition 2.13 of \cite{degensphermod} shows that $Z(G)$ is in the kernel of $\widehat{\psi}$ and that the isomorphism 
\begin{equation}
H^0(G \cdot x_0,\shN_{X_0|V})^G \to \Vgg, s \mapsto s(x_0)
\end{equation} in \eqref{eq:14} above is indeed $\Tad$-equivariant if $\Vgg$ is equipped with the $\Tad$-action $\alpha$ and $H^0(G \cdot x_0,\shN_{X_0|V})^G$ is equipped with the $\Tad$-action $\widehat{\psi}$.

In Section~\ref{sec:extending-sections} we strengthen
\cite[Proposition 3.4]{degensphermod} and obtain necessary and
sufficient conditions for a section $s \in  H^0(G \cdot x_0,
\shN_{X_0|V})^G$ to extend to $X_0$: see
Corollary~\ref{cor:extending-sections}. The proof is given in Section~\ref{subsubs_pf_of_extension_thm}, after we review some generalities about extending sections of a vector bundle over a normal variety in Section~\ref{subsubs!generalities_about_extending}.
In Section~\ref{sec:few-more-facts-tgms} we gather a few more results
on $\tg\ms$. 

\subsection{Extending sections} \label{sec:extending-sections}

We denote by  $X_0^{\leq 1}  \subset X_0$ 
the union of $G \cdot x_0$ with all $G$-orbits of $X_0$ that have
codimension $1$.  By \cite [Lemma 1.14] {brion-cirmactions} 
$X_0^{\leq 1}$ is an open subset of $X_0$, and because $X_0$ is
normal, it is a subset of the smooth locus of $X_0$ (see, e.g., the argument
in the proof of \cite  [Lemma 3.3]{degensphermod} for details). 

\begin{defn}
We say the $\lambda \in E$ \textbf{has codimension one} if 
\[\dim G \cdot (x_0 - v_{\lambda}) = (\dim G\cdot x_0) - 1.\]   
\end{defn}

As an immediate consequence of, e.g., \cite[Proposition 3.1]{degensphermod}
one has the following simple criterion to determine whether an
element of $E$ has codimension one.
\begin{prop} \label{prop:lambdacodim1}
For $\lambda \in E$ the following are equivalent
\begin{enumerate}
\item $\lambda$ has codimension one;
\item for every $\alpha \in \sr$ such that $\<\alpha^\vee, \lambda\>
  \neq 0$, there exists $\mu \in E\setminus \{\lambda\}$ such that $\<\alpha^\vee, \mu\>
  \neq 0$;
\item for every positive root $\alpha$ such that $\<\alpha^\vee, \lambda\>
  \neq 0$, there exists $\mu \in E\setminus \{\lambda\}$ such that $\<\alpha^\vee, \mu\>
  \neq 0$.
\end{enumerate}  
\end{prop}

The following is an immediate consequence of \cite[Lemma 3.9]{brion-ihs}.
\begin{prop} \label{prop:extendcodim1}
If $s \in  H^0(G \cdot x_0,
\shN_{X_0|V})^G$, then the following are equivalent:
\begin{enumerate}[(1)]
\item  $s$ extends to $X_0$; \label{item:1}
\item $s$ extends to $X_0^{\leq 1}$; \label{item:2}
\item $s$ extends to  $G\cdot x_0 \cup G \cdot (x_0-v_{\lambda})$ for every $\lambda \in E$ of
codimension $1$.  \label{item:3}
\end{enumerate}
\end{prop}

\begin{proof}
The equivalence of (\ref{item:1}) and (\ref{item:2}) is a special case of
\cite[Lemma 3.9]{brion-ihs}. The equivalence of (\ref{item:2}) and
(\ref{item:3}) is a consequence of the definition of a
sheaf, once we prove that the collection of sets 
\begin{equation}
\{G\cdot x_0 \cup G \cdot
(x_0-v_{\lambda}) \colon \lambda \in E\text{ of codimension $1$}\} \label{eq:X0opencover}
\end{equation} 
forms an
open cover of $X_0^{\leq 1}$. 
We first show that each set $G\cdot x_0 \cup G \cdot
(x_0-v_{\lambda})$ in the collection~\eqref{eq:X0opencover} is open. Indeed, the complement of $G\cdot x_0 \cup G \cdot (x_0-v_{\lambda})$ in $X_0^{\leq 1}$ is a finite (by Remark~\ref{rem:propsX0}(\ref{item:X0orbits})) union of orbits in $X_0^{\leq 1}$ that are all closed because they are of minimal dimension. Secondly, the union of the sets in the collection~\eqref{eq:X0opencover} is all of $X_0^{\leq 1}$ because, by Remark~\ref{rem:propsX0}(\ref{item:X0orbits}) and \cite[Lemma 2.16]{degensphermod} every orbit of codimension $1$ in $X_0$ is of the form 
$G \cdot (x_0-v_{\lambda})$ for some $\lambda \in E$ of codimension $1$. 
\end{proof}

We recall some well-known facts about $\Tad$-weights and $\Tad$-eigenvectors in $\Vgg$.
\begin{lemma}
Let $\beta \in \wl$. If $v\in V$ such that $[v] \in \Vgg$ is a nonzero $\Tad$-eigenvector of $\Tad$-weight $\beta$, then there exists a $\Tad$-eigenvector $\widehat{v}\in V$ of $\Tad$-weight $\beta$ such that $[v] = [\widehat{v}] \in \Vgg$.
\end{lemma}
\begin{proof}
This follows from the following standard argument. Note that $\Vgg$ is a $\Tad$-stable subspace of $\Vg$. Moreover, since the subspace $\fg\cdot x_0$ of $V$ is $\Tad$-stable, there exists another $\Tad$-stable subspace $L$ of $V$ such that $V = \fg\cdot x_0 \oplus L$. The restriction of the quotient map $V \to \Vg$ to $L$ is an isomorphism $L \to \Vg$ of $\Tad$-modules. We can take $\widehat{v}$ to be the inverse image in $L$ of $[v]$ under this isomorphism.  
\end{proof}

\begin{prop} \label{prop:propsVggweights}
Let $\beta \in \wl$ and assume that $v$ is a $\Tad$-eigenvector in $V$ of weight $\beta$ such that $0 \neq [v] \in \Vgg$. Then
\begin{enumerate}[(a)]
\item there exists $\alpha \in \sr$ such that $X_{\alpha}\cdot v \neq 0$ and $\beta -\alpha \in R^+ \cup\{0\}$; \label{item:betasum}
\item $\beta \in \<E\>_{\ZZ}$; \label{item:betalattice}
\item $X_{\alpha}\cdot v \in \<X_{-(\beta-\alpha)}\cdot x_0\>_{\CC}$ for all $\alpha \in \sr$ such that $\beta-\alpha \in R^{+}$; \label{item:betaminalpharoot}
\item $X_{\alpha}\cdot v=0$ for all $\alpha\in \sr$ such that $\beta - \alpha \notin R^+ \cup\{0\}$.
\label{item:betaminalphanotroot}
\end{enumerate}
\end{prop}
\begin{proof}
Assertion (\ref{item:betalattice}) is a consequence in Lemma 2.17(c) of \cite{degensphermod}. Let $\alpha \in \sr$. 
Recalling that if $X_{\alpha}\cdot v$ is nonzero, then it has $\Tad$-weight $\beta - \alpha$,
assertions (\ref{item:betasum}) and (\ref{item:betaminalphanotroot}) follow from Lemma 2.18 in \loccit\
Since the root operator $X_{\alpha}$ belongs to the Lie algebra of $G_{x_0}$, we have that $X_{\alpha}\cdot v \in \fg \cdot x_0$. Assertion (\ref{item:betaminalpharoot}) now follows from the fact that if the $\Tad$-weight $\beta-\alpha$ occurs in $\fg\cdot x_0$ then the corresponding weight space is $\<X_{-(\beta-\alpha)}\cdot x_0\>_{\CC}$.         
\end{proof}

\begin{remark}
Since we will frequently make use of it later, we note the following consequence of (\ref{item:betasum}) and (\ref{item:betalattice}) in Proposition~\ref{prop:propsVggweights}: if $\beta$ is a $\Tad$-weight in $\Vgg$ then
\begin{equation}
0 \neq \beta \in \<\sr\>_{\NN} \cap \<E\>_{\ZZ}.  \label{eq:betasumm}
\end{equation}
\end{remark}

\begin{theorem} \label{thm:extension}
Assume that $v$ is a $\Tad$-eigenvector in $V$ of $\Tad$-weight $\beta$  such that 
 \[
                 0 \not=   [v] \in  \Vgg.
\]  
Denote by   $s \in  H^0 ( G \cdot x_0, \shN_{X_0|V} )^G$   the $G$-equivariant
section defined by $s(x_0) = [v]$. Let $\lambda$ be an element of $E$
which has codimension $1$ and put $Z = G\cdot x_0 \cup G \cdot (x_0 -
v_{\lambda})$. Let $a$ be the coefficient of $\lambda$ in the unique
expression of $\beta$ as a linear combination of elements of $E$.
\begin{enumerate}[A)]  
\item If  $a \leq 0$, then  $s$ extends to an element of $H^0 (Z,
  \shN_{X_0|V} )^G$.  \label{item:extA}
\item If  $a > 1$,   then $s$  does not extend  to an element of $H^0
  (Z, \shN_{X_0|V} )^G$. \label{item:extB}
\item Assume   $a = 1$.  Then the following are equivalent: \label{item:extC}
\begin{enumerate}[i)]
\item $s$  extends to an element of $H^0 (Z, \shN_{X_0|V} )^G$;
 \item There exists  $\hat{v}  \in V(\lambda)$  such that  
                         $[v] = [\hat{v}]$  as elements of  $\Vg$.
\end{enumerate}
\end{enumerate}
\end{theorem}

The proof of Theorem~\ref{thm:extension} will be given in Section~\ref{subsubs_pf_of_extension_thm}
which starts on page~\pageref{subsubs_pf_of_extension_thm}. Before
that, we 
gather some general results on extending sections in Section~\ref{subsubs!generalities_about_extending}. The following is a synthesis of Proposition~\ref{prop:extendcodim1}
and Theorem~\ref{thm:extension}.
\begin{cor} \label{cor:extending-sections}
Assume that $v$ is a $\Tad$-eigenvector in $V$ of $\Tad$-weight $\beta$  such that 
 \[
                 0 \not=   [v] \in  \Vgg.
\]  
Denote by   $s \in  H^0 ( G \cdot x_0, \shN_{X_0|V} )^G$   the $G$-equivariant
section defined by $s(x_0) = [v]$. Let 
\begin{equation}
\beta  = \sum_{\lambda \in E} a_{\lambda} \lambda  \label{eq:beta}
\end{equation}
be the unique expression of $\beta$ as a $\ZZ$-linear combination of
the elements of $E$. 

The section $s$  extends to an element of $H^0 (X_0, \shN_{X_0|V} )^G$
if and only if for all $\lambda \in E$ of codimension $1$ we have
\begin{itemize}
\item[-] $a_{\lambda} \leq 1$; and
\item[-] if $a_{\lambda} = 1$ then there exists  $\hat{v}  \in V(\lambda)$  such that  
                         $[v] = [\hat{v}]$  as elements of  $\Vg$.
\end{itemize}
\end{cor}

\begin{remark}
It follows from Proposition~\ref{prop:Kostant} below that if $s$ extends, then at most two of the $a_{\lambda}$ in equation~\ref{eq:beta} are positive, irrespective of whether $\lambda$ is of codimension $1$ or not. 
\end{remark}

\begin{example}[Luna]   \label {ex!lunas_example}  
 Let $G = \SL_2 \times \SL_2$, and $E = \{\lambda_1, \lambda_2\}$ with  
 $ \lambda_1 = 2\omega,  \lambda_2    = 4\omega + 2 \omega'$. Here $\omega$ is the fundamental weight of the first component of $G$, and $\omega'$ is that of the second component. Similarly, we will use $\alpha$ and $\alpha'$ for the simple root of the first and second component of $G$, respectively. 
Using Proposition~\ref{prop:lambdacodim1},  it follows that  $v_{\lambda_2}$ has a $G$-orbit of codimension $1$ in $X_0$, 
while $v_{\lambda_1}$ has a $G$-orbit of codimension $\geq 2$. Hence by Proposition~\ref{prop:extendcodim1} 
 for $[w] \in \Vgg$, 
the equivariant section of $\shN_{X_0|V}$ on $G \cdot x_0$ induced by $[w]$ extends to $X_0$ if and
only if it extends over $G \cdot x_0 \cup G \cdot v_{\lambda_2}$. 

Denote by $e_1,e_2$  (resp.\ $g_1,g_2$)  a basis
of $\CC^2$ where the first (resp.\ second) $\SL_2$ acts in the standard fashion.
A small calculation
gives that  the vector space $ \Vgg$  is $3$-dimensional with basis
the classes in  $V / \fg \cdot x_0$ of 
 \[
     w_1=  e_1e_2 ,  \quad   \quad  w_2 = e_1^4 \otimes g_2^2 ,\quad    \quad
            w_3 = e_2^2 + e_1^2e_2^2 \otimes g_1^2. 
\]
The vector $w_1$ has $\Tad$ weight
$\alpha = \lambda_1$, and since $w_1 \in V(\lambda_1)$
part \ref{item:extC}) of Theorem~\ref{thm:extension} 
implies that the induced
equivariant section extends to $G \cdot x_0 \cup  G \cdot v_{\lambda_2}$, hence
to the whole of $X_0$.  
The vector $w_2$ has $\Tad$ weight
$2\alpha' =  - 4\lambda_1 + 2\lambda_2 $,
and part \ref{item:extA}) of Theorem~\ref{thm:extension} 
implies that the induced
equivariant section extends to $X_0$.  
The vector $w_3$ has $\Tad$ weight $2\alpha = 2\lambda_1$, 
hence part \ref{item:extB}) of Theorem~\ref{thm:extension} 
implies that the induced
equivariant section does not extend to $G \cdot x_0 \cup  G \cdot v_{\lambda_2}$.
We have shown that
\[\tg\ms = V(\alpha) \oplus V(2\alpha')\]
as a $\Tad$-module.
We remark that to exclude the section induced by $w_3$ from $\tg\ms$ we could not have used 
\cite[Proposition 3.4]{degensphermod}, since condition (ES2) of that proposition
is not satisfied for $w_3$.   

We also remark that Luna has shown in an unpublished note from 2005 that this moduli scheme $\ms$, equipped with its reduced scheme structure, is a union of two affine lines meeting in a point. It was the first example of a non-irreducible scheme $\ms$. 
\end {example}

\subsection{Generalities about extending sections} \label{subsubs!generalities_about_extending}

In this section  $X$  denotes a variety, in particular it is reduced,
irreducible and separated. Let $\VE \to X$ be an algebraic vector bundle.
For the proof of Theorem~\ref{thm:extension} in
Section~\ref{subsubs_pf_of_extension_thm}
we need the following general propositions. They are well known, but 
for completeness we provide proofs.

\begin {prop}  \label{prop!about_poles_of_rational_functions}   
Assume that $X$ is normal,
 that $U \inn X$ is a nonempty Zariski open subset, and that $f : U \to \CC$ a morphism. If $f$
does not extend to a morphism $X \to \CC$, then there exists $p \in X \setminus U$
such that for every irreducible algebraic curve $C \inn X$ with $p \in C$ and $U \cap C \not= \emptyset$ 
the morphism $f$ restricted to $U \cap C$ does not extend to a morphism $C \to \CC$.
\end{prop}

\begin {proof} We consider  $f$ as a rational function on $X$.
Since $f$ does not extend to a morphism $X \to \CC$ it follows
that $f$ is not the constant function with value $0$. 
Using that $X$ is normal, there is a well defined divisor of poles of $f$ and
a well defined divisor of zeroes of $f$, see, e.g. \cite[Section 4.0]{cox_little_schenck-toricbook}.   Since $f$ does not
extend to a morphism $X \to \CC$  the divisor of poles of $f$ is
nonzero, see, e.g. \cite[Propopsition 4.0.16]{cox_little_schenck-toricbook}.

We fix a point  $p \in X$ which is in
the support of the divisor of poles of $f$ but not in the support of the  divisor of zeroes.
To obtain a contradiction, we assume that there exists an  irreducible algebraic curve $C \subset X$ with $p \in C$ and 
$U \cap C \not= \emptyset$  such that 
the morphism $f$ restricted to $U \cap C$ extends to a morphism 
$\bar{f} : C \to \CC$. 

Denote by $g$ the rational function  $1/f$. 
The divisor of zeroes of $g$ is the divisor of poles of $f$
and the divisor of poles of $g$ is the divisor of zeroes of $f$.
Hence $p$ is in the support of the divisor of zeroes of $g$
and is not in the support on the divisor of poles of $g$.
It follows that there exists a Zariski 
open subset $W \subset X$ with $p \in W$ such that $g$
defines a morphism $g : W \to \CC$ with the property $g(p) = 0$.
Denote by $\bar {g} : W \cap C \to \CC$ the restriction of $g$ to
$W \cap C$.  We have that $W \cap U $ is
a nonempty Zariski open subset of $X$. Since
$W \cap C \not= \emptyset$ and $U \cap C \not= \emptyset$ 
we get that $W \cap U \cap C \not= \emptyset$. For 
$q \in W \cap U \cap C$ we have $(\bar{f} \bar{g}) (q)= (fg) (q) = 1 $.
Since $W \cap U \cap C$ is dense in $C$ it follows that 
$(\bar{f} \bar{g}) (p) = 1$
which  contradicts $\bar{g} (p) = 0$.  \end {proof}

\begin {example}  If $X = \CC^2$ and $f = x/y$ we can choose $p = (a,0)$
for any $a \in \C \setminus \{0\}$.
\end {example}

\begin {prop}  \label{prop!about_poles_of_sections}   
Assume that $X$ is normal, that $U \subset X$ is a nonempty Zariski
open subset, and that  $s \in H^0(U,\VE)$  is a section of the
vector  bundle $\VE$.
If $s$ does not extend to a section $X \to \VE$, then there exists $p \in X \setminus U$
such that for every algebraic curve $C \subset X$ with $p \in C$ and $U \cap C \not= \emptyset$ 
the section $s$ restricted to $U \cap C$ does not extend to a section  $C \to \VE$.
\end{prop}

\begin {proof}  By the defining gluing property of sections of
  sheaves, there exists a nonempty Zariski open $V \subset X$ such that 
$\VE$ restricted to $V$ is trivial and $s$ restricted to $V \cap U$
does not extend to a section  over $V$. Hence we can assume, without loss of generality,
that $\VE$ is 
the trivial vector bundle.  So assume  $\VE = X \times \CC^n \to X$ is the first
projection. Let $e_1, \dots ,e_n$ be a basis of $\CC^n$ and define
$s_i  \in H^0 (X,\VE)$ by  $s_i (x) = (x,e_i)$ for all $x \in X$. There exist
(unique) morphisms $f_i : U \to \CC$ such that 
\[
      s(u) =  \sum_{i=1}^n  f_i(u) s_i(u)
\]
for all $u \in U$.   If each $f_i$ extended to a morphism $X \to \CC$,  it would then
follow that $s$ extends to a section $X \to \VE$ which contradicts the assumptions.
Hence at least one of the $f_i$ does not extend to a morphism. Using
Proposition~\ref{prop!about_poles_of_rational_functions} 
there exists $p \in X \setminus U$
such that for every algebraic curve $C \subset X$ with $p \in C$ and $U \cap C \not= \emptyset$,
the morphism $f_i$ restricted to $U \cap C$ does not extend to a morphism  $C \to \CC$. 
As a consequence,  $s$ restricted to $U \cap C$ does not extend to 
a section $C \to \VE$.
\end {proof}

Assume now in addition that $G$ is a connected linear algebraic group over $\CC$, $X$ is a  $G$-variety
and $\pi : \VE \to X$ is a $G$-vector bundle over $X$.  This means that we are given an algebraic action 
$\rho : G \times \VE \to\VE$ such that  
\[ 
     \rho ( g,v) \in \pi^{-1}  ( g \cdot (\pi(v))) 
\]
for all $g \in G, v \in V$ and that for fixed $g \in G$ and $x \in X$ the induced map
\[
     \pi^{-1} (x)  \to \pi^{-1} ( g \cdot x),  \quad \quad   v \mapsto  \rho (g , v)
\]
is an isomorphism of vector spaces.

While its proof is elementary, the following proposition implies that the
section $s$ of Theorem~\ref{thm:extension} extends to $Z$ if and
only if it extends along just one curve; see
Proposition~\ref{prop!sextendsifandonlyifsadoes}.

\begin {prop}  \label{prop!section_extension_for_G_bundles}   Assume that $X$ is normal,
 that $U \subset X$ is a nonempty  G-stable Zariski open subset such that $X \setminus U$
is a single $G$-orbit, and that $s \in H^0(U,\VE)^G$  is a $G$-equivariant
section  $U \to \VE$.   Assume that there exists $p_0 \in X \setminus U$ and 
 an algebraic curve  $C_0 \subset X$   with $p_0 \in C_0$ and $U \cap C _0\not= \emptyset$ 
such that $s$ restricted to $U \cap C_0$ extends to a section $s_0 : C_0 \to \VE$.
Then  $s$  extends to an element of $H^0(X,\VE)^G$.
\end{prop}

\begin{proof}  We first show that $s$  extends to an element of $H^0(X,\VE)$, and then that 
the extension is $G$-equivariant.
 
 We  assume that $s$ does not extend to an element of $H^0(X,\VE)$, and we will get a contradiction.
By  Proposition~ \ref{prop!about_poles_of_sections}  there exists 
$p \in X \setminus U$ such that for every algebraic curve $C \subset X$ with $p \in C$ and $U \cap C \not= \emptyset$ 
the section $s$ restricted to $U \cap C$ does not extend to a section  $C \to \VE$.   Since
$X \setminus U$ is a single $G$-orbit, there exists $g \in G$ with $g \cdot p_0 = p$. 

 Set $  C = \{ g \cdot v :   v \in C_0 \} $  and define $t : C \to \VE$ by  
\[
     t(v)  =  g \cdot s_0 (g^{-1} \cdot v)
\]
for $v \in C$.  Since $s$ is $G$-equivariant we have $t|_{U \cap C} = s|_{U \cap C}$, hence $t$
is a section of $\VE$ over $C$ which extends $s|_{U \cap C}$, contradicting the choice of $p$.

We have shown that $s$ extends to a section $s_1 : X \to \VE$. We claim that $s_1$ is $G$-equivariant.
Indeed, define $s_2 :  X \to \VE$ by   $ s_2 (v)  = g \cdot s_1 (g^{-1} \cdot v)$.   Since $s$ is $G$-equivariant on $U$,
we have that   $s_2(u) = s(u) = s_1(u)$ for all $u \in U$.  As a consequence $s_2 = s_1$, which
implies that $s_1$ is $G$-equivariant.
\end{proof}

\subsection {Proof of Theorem~\ref{thm:extension}}  \label{subsubs_pf_of_extension_thm}

We start the proof of Theorem~\ref{thm:extension}. Let $\lambda \in E$
be of codimension $1$. For $t \in \C$, we put
\begin{equation}
z_t := t\cdot v_{\lambda} + \sum_{\mu \in E \setminus \{\lambda\}} v_{\mu}.
\end{equation}
Note that, because $E$ is linearly independent, $z_t \in T \cdot x_0
\inn G\cdot x_0$ for $t \in \C\setminus\{0\}$. Moreover,
\begin{equation}
z_0 = \sum_{\mu \in E \setminus \{\lambda\}} v_{\mu} = x_0 -
v_{\lambda};
\end{equation}
and $G \cdot z_0$ has codimension $1$ in $X_0$.

Since $Z=G\cdot x_0 \cup G \cdot z_0$ is smooth, the restriction  of
the sheaf  $\shN_{X_0|V}$ to $Z$ is locally free. We denote by
$ \VE  \to Z$ the total space of the corresponding vector bundle. In
particular the sections of the restriction of $\shN_{X_0|V}$ to $Z$
are naturally identified with those of $\VE$. 

Recall that   $s \in  H^0 ( G \cdot x_0, \VE )^G$  denotes the
equivariant section induced by $v$; that is $s(x_0) = [v] \in \Vgg$. 
Set $C_0 = \{ z_t :  t \in \CC \}$, and denote by $ s^{*} $ the section
of $\VE$ over  $G \cdot x_0  \cap C_0$ defined by $ s^* (z_t) = s(z_t)$.

\begin {prop}  \label {prop!sextendsifandonlyifsadoes} 
 The section  $s$ extends to an element of $H^0 (Z, \VE)^G$ if and only if
 $s^*$ extends to a section of $\VE$ over $C_0$.
\end {prop}

\begin {proof}  If $s^*$ extends, then so does $s$ by Proposition~\ref{prop!section_extension_for_G_bundles}.
   The converse is obvious.
\end {proof}

For $w \in V$ we denote by $s_w \in H^0( X_0, \shN_{X_0|V} )$ the global
section defined by
\[
     s_w (x) = [w] \in V/T_{x} X_0 
\]
for all $x \in X_0$. We will use $V^{\beta}$ for the $\Tad$-weight space in $V$ of weight
$\beta$.  Recall that, by Proposition~\ref{prop:propsVggweights}, $\beta$ is a nonzero
element of $\<\sr\>_{\NN}$ and that $v \in V^\beta$ such that $0 \neq
[v] \in \Vgg$. 

The idea of the proof of Theorem~\ref{thm:extension} is to find elements
$\{y_i\}$ in $\Vbeta$ such that their images in $V/T_{z_t}X_0$
form a basis of the $\Tad$-weight space of weight $\beta$ in $V/T_{z_t}X_0$ for \emph{all} $t \in \C$. It then follows that 
(the restriction to $C_0$ of) the sections $s_{y_i}$ form a linearly independent subset of
$H^0(C_0, \VE)$, and that there exist $f_i \in \C(t)$
such that for all $t \in \C \setminus \{0\}$ we have
\begin{equation}
s^*(z_t) = \sum_i f_i(t) s_{y_i}(z_t).
\end{equation}
 The section $s^*$ extends to all of $C_0$ if and only if each $f_i(t)$
 belongs to the polynomial ring $\C[t]$. With the appropriate
 choice of the vectors $\{y_i\}$ the rational functions
 $f_i(t)$ are very simple; see Proposition~\ref{prop!s_interms_of_syi}.

By \cite[Lemma 3.3]{degensphermod}  $T_{z_0} X_0$ is the linear span 
of $\fg \cdot z_0 \cup \{ v_{\lambda} \}$.  If $t$ is nonzero, then $G
\cdot z_t$ is open in $X_0$, whence  $T_{z_t} X_0 = \fg \cdot z_t$, and 
$ v_{\lambda} \in T_{z_t} X_0$, by Remark~\ref{rem:propsX0}(\ref{item:t0x0}).  The image  of $\Vbeta$  under the projection $V \to V/T_{z_t} X_0$
can naturally be identified with $\Vbeta/ (\Vbeta \cap T_{z_t} X_0)$.  

\begin {lemma}  \label{lem:image_of_Vbeta}
Assume $t \in \CC$.
If $\beta$ is not a root, then   $\Vbeta \cap T_{z_t} X_0 = 0$.
If $\beta$ is a root, then $\Vbeta \cap T_{z_t} X_0$ is equal to $\<X_{-\beta} \cdot z_t\>_{\CC}$, so it is either $0$ or $1$-dimensional.
\end {lemma}

\begin {proof}
Recall that  $T_{z_t} X_0$ is equal to the
linear span of $\fg \cdot z_t \cup \{ v_\lambda \}$.
Using that $\fu^{-}$ is spanned by the set $\{X_{-\gamma} \colon \gamma
\text{ positive root of } G\}$, the lemma follows from the fact that $v_\lambda$ has $\Tad$-weight
zero (and so not equal to $\beta$) and that
\(\fg \cdot z_t = \ft \cdot z_t \oplus \fu^{-} \cdot z_t.\)
\end {proof}

Before giving the details of the remaining arguments for the proof
of Theorem~\ref{thm:extension}, we introduce some more notation for
the remainder of this section. Put 
\[V_1:=V(\lambda); \quad \quad \quad V_2:= \oplus_{\mu \in
  E\setminus\{\lambda\}} V(\mu).\]
Note that $V=V_1 \oplus V_2$.   
Set $n = \dim \Vbeta+ 1 $ and $m = n - \dim \Vbeta \cap V_1$.

Because the summands $V_1$ and $V_2$ of $V$ are stable under
the $\Tad$-action, there exists a basis $y_2, \dots , y_n$ of $\Vbeta$, 
such that $y_i \in  \Vbeta \cap V_2$
for $2 \leq i \leq m$  and $y_i \in   \Vbeta \cap V_1$ for $m+1 \leq i \leq n$.
Since $v \in \Vbeta$, there exist $b_i \in \CC$ such that
\begin {equation} \label{eqn!dvintermsofyi}
     v = \sum_{i=2}^n b_i y_i.
\end {equation} 

Recall that $a$ is the coefficient of $\lambda$ in the unique expression
of $\beta$ as a linear combination of the elements of $E$. 
\begin {prop}  \label{prop!s_interms_of_syi}
Let $\{y_2, y_3, \ldots, y_n\}$ be a basis of $V^\beta$ as above, and
let $b_i$ be elements of $\C$ such that the
equality~(\ref{eqn!dvintermsofyi}) holds.
For all $t \in \CC^*$ we have 
\[
     s (z_t) = t^{-a} (\sum_{i=2}^m b_i s_{y_i}(z_t))  +
                 t^{-a+1}  (\sum_{i=m+1}^n b_i s_{y_i}(z_t)).
\]
\end {prop}

\begin {proof}  
By assumption, $s$ is an eigenvector of weight $\beta$ for the $\Tad$-action $\widehat{\psi}$ on $H^0(G\cdot x_0, \VE)^G \cong \Vgg$. Equivalently, $s$ is an eigenvector (of weight described below) for the natural action of $\GL(V)^G$ on $H^0(G\cdot x_0, \VE)^G$ described in \cite[p.~1777]{degensphermod} and recalled at the start of Section~\ref{sec:crit-extens} above.

Set $D=E\setminus \{\lambda\}$ and recall the map  $\sigma_D : \CC^{\times} \to \GL(V)^G$
of \cite[p.~1784]{degensphermod}: for $t \in \CC^{\times}$, the element $\sigma_D(t)$ of $\GL(V)^G$ is defined by  $\sigma_D (t) \cdot (w_1 + w_2) = tw_1 + w_2$, for all
$w_1 \in V_1$ and $w_2 \in V_2$. 

We now argue as in the proof of Part i) of  \cite [Proposition 3.4]{degensphermod}. The homomorphism $f:T \to \GL(V)^G$ of \eqref{eq:homomf} is surjective (because $E$ is linearly independent), and therefore the homomorphism $f^*: X(\GL(V)^G) \to X(T), \delta \mapsto \delta \circ f$ of character groups is injective. Then $\delta:=(f^*)^{-1}(\beta)$ is the $\GL(V)^G$-weight of $s$. 
Consequently
\begin{align*}
s(z_t) = s(\sigma_D(t)\cdot x_0) &= \delta(\sigma_D(t))^{-1}\sigma_D(t)\cdot s(x_0)\\
&= \delta(\sigma_D(t))^{-1} [\sigma_D(t)\cdot v] \in V/\fg\cdot z_t.
\end{align*}
Since, by definition, $a$ is the coefficient of $\lambda$ in the expression of $\beta$ as a $\ZZ$-linear combination of the elements of $E$, we have that $\delta(\sigma_D(t)) = t^a$. Consequently
\[s(z_t) =  t^{-a}[\sigma_D(t) \cdot v].\]
Taking into account 
      that  
      \[
          \sigma_D (t)\cdot v =  \sum_{i=2}^m b_i y_i  + t \sum_{i=m+1}^n b_i y_i. 
       \] 
the proposition follows.
\end {proof}

We prove part \ref{item:extA}) of Theorem~\ref{thm:extension}.
Assume $a \leq 0$.  
Since the  $s_{y_i}$ are sections defined over the whole $X_0$, 
Proposition~\ref {prop!s_interms_of_syi} implies that $s^*$ extends to
a section over $C_0$. Hence $s$ extends to $Z$ by 
Proposition~\ref {prop!sextendsifandonlyifsadoes}.

For the rest of the proof we separate into five cases.
Recall that $x_0 = \sum_{\lambda \in E} v_{\lambda}$.

\begin{itemize}
\item[Case 1:] $\beta$ is a root,  $X_{-\beta} \cdot v_{\lambda} \not= 0$ 
       and there exists $\mu \in E\setminus \{\lambda\}$  with $X_{-\beta} \cdot v_{\mu} \not= 0$.
\item[Case 2:]  $\beta$ is a root,  $X_{-\beta} \cdot v_{\lambda} = 0$ 
       and there exists $\mu \in E\setminus \{\lambda\}$  with $X_{-\beta} \cdot v_{\mu} \not= 0$. 
\item[Case 3:]  $\beta$ is a root, $X_{-\beta} \cdot v_{\lambda} \not= 0$ 
            and $X_{-\beta} \cdot v_\mu = 0$  for all  $\mu \in E\setminus \{\lambda\}$.  
\item[Case 4:]   $\beta$ is a root, $X_{-\beta} \cdot v_{\lambda} = 0$ 
            and $X_{-\beta} \cdot v_\mu = 0$  for all  $\mu \in E\setminus \{\lambda\}$.  
\item[Case 5:]  $\beta$ is not a root. 
\end{itemize}

We first show that Case 3 and Case 4 cannot happen. 
Case 3 cannot occur, because it contradicts the assumption that
$\lambda$ has codimension $1$, by Proposition~\ref{prop:lambdacodim1}. 
Case 4 cannot occur, because $\beta
\in \<E\>_{\ZZ}$ by Proposition~\ref{prop:propsVggweights}. Indeed, if $\beta$ is a root in $\<E\>_{\ZZ}$, then $\{\mu \in
E \colon\<\beta^{\vee}, \mu\> \neq 0\}$ is nonempty and so
$X_{-\beta} \cdot x_0 \neq 0$. 

We now prove \ref{item:extB}) and \ref{item:extC})  of Theorem~\ref{thm:extension} in
Case 1. We begin by choosing an appropriate basis of $\Vbeta$. Put $y_n :=
X_{-\beta}\cdot v_{\lambda}$, and let $y_2, y_3, \ldots, y_q$ be the
elements of 
\[\{ X_{-\beta} \cdot v_{\mu} \colon \mu \in E\setminus\{\lambda\},
\<\beta^\vee, \mu\> \neq 0\}\]
in some order. Finally, extend $y_2, y_3, \ldots, y_q, y_n$ to a basis
$y_2, y_3, \ldots, y_n$ of $\Vbeta$ such that  $y_i \in V_2$ when $q+1 \leq i \leq m$ and
$y_i \in   V_1$ when $m+1 \leq i \leq n-1$.

Assume $t \in \CC$. Since
$\fg \cdot z_t \subset T_{z_t} X_0$, we have 
$X_{-\beta} \cdot z_t \in T_{z_t} X_0$.  Hence there is, in $V/T_{z_t} X_0$,
the following equality 
\begin{equation}
    [y_2] =  - \sum_{i=3}^q [y_i] - t [y_n].  \label{eq:15}
\end{equation}

Using Lemma~\ref {lem:image_of_Vbeta}   it follows that  the classes,
in $V/T_{z_t} X_0$,  of $y_3, \dots ,y_n$ are a basis
for the image of $\Vbeta$ in $V/T_{z_t} X_0$. 
In other words,   the elements
$s_{y_i}(z_t)$, for $3 \leq i \leq n$,  are linearly independent for
every $t \in \C$. 

Combining the relation~(\ref{eq:15}) with Proposition~\ref
{prop!s_interms_of_syi} we get for all nonzero $t$ that
\begin {multline}  \label {eqn_key_lin_independent_formula}
          s (z_t) = t^{-a} (\sum_{i=3}^q (b_i-b_2) s_{y_i}(z_t))   +
                    t^{-a} (\sum_{i=q+1}^m b_i s_{y_i}(z_t))   + \\ 
                 t^{-a+1}  (\sum_{i=m+1}^{n-1} b_i s_{y_i}(z_t))  +
                 t^{-a+1}  (b_n-b_2) s_{y_n}(z_t).
\end {multline}

We now prove part \ref{item:extB}) in Case 1. Assume 
 $a > 1$. We assume that $s$ extends and we will get
a contradiction. Since $s$ extends we have that $s^*$ also extends.
Since the set $\{s_{y_i}(z_t)\colon 3 \leq i \leq n\}$ is linearly independent for
all $t  \in \CC$,  and $-a +1$ and $-a$ are negative,
Equation~(\ref {eqn_key_lin_independent_formula})
implies that  $b_i = b_2$ for  $3 \leq i \leq q$ and for $i = n$, and that
$b_i = 0$ for  $ q+1 \leq i \leq n-1$. Hence $v = b_2(X_{-\beta} \cdot x_0)$,
contradicting the assumption $ v \notin \fg\cdot x_0$.
This proves part \ref{item:extB}) in Case 1.

We now prove part \ref{item:extC}) in Case 1.
Assume that $a = 1$.  Since $-a+1 =0$ and $-a < 0$, 
arguing similarly to the case $a > 1$   we get
 that $s^*$ extends 
over $C_0$ if and only if $b_i = b_2$ for  $3 \leq i \leq q$ 
and $b_i = 0$ for  $ q+1 \leq i \leq m$. If these conditions
hold then  $v - \hat{v} \in \fg \cdot x_0$, where
\[
    \hat{v}  =  \sum_{i=m+1}^{n-1} b_i y_i + (b_n-b_2)y_n,
\]
which is an element of $V_1$. Conversely, assume there exists $\hat{v} \in V_1$
with  $v - \hat{v} \in \fg \cdot x_0$. Then $v$ and $\hat{v}$ define the same
equivariant section $s$ of $\VE$ over $G \cdot x_0$. Hence we can
assume in Equation (\ref {eqn!dvintermsofyi}) that $b_i = 0 $ for
$ 2 \leq i \leq m$. As a consequence, Proposition~\ref {prop!s_interms_of_syi}
implies that $s^*$ extends, hence by  
Proposition~\ref{prop!sextendsifandonlyifsadoes}  $s$ also extends.
This finishes the proof  of part  \ref{item:extC})  of
Theorem~\ref{thm:extension} in Case 1.

The arguments for the proof  of \ref{item:extB}) and \ref{item:extC})  of Theorem~\ref{thm:extension} for 
Cases 2 and 5 are very similar to those  of Case 1 so we only sketch them. 

In Case 2, we can choose a basis  $y_2, y_3, \ldots , y_n$ of $\Vbeta$ with the
following properties: 
\begin{itemize}
\item[-] $\{y_2, y_3, \ldots y_q\} = \{X_{-\beta} \cdot v_{\mu} \colon
  \mu \in E, \<\beta^{\vee}, \mu\> \neq 0\}$;
\item[-] $y_i \in V_2$ for $q+1 \leq i \leq m$; and
\item[-] $y_i \in  V_1$ for $m+1 \leq i \leq n$.
\end{itemize}
Then, for all $t \in \CC$  we have the
following equality in   $V/T_{z_t} (X_0)$:
\[
    [y_2] =  - \sum_{i=3}^q [y_i] 
\]
and, similarly to Case 1,
the elements
$s_{y_i}(z_t)$, for $3 \leq i \leq n$,  are linearly independent.  
Continuing as in the proof of Case 1 the result follows.

In Case 5,  we choose a basis   $y_2, \dots , y_n$ of $\Vbeta$ with
the following properties:
\begin{itemize}
\item[-]$y_i \in V_2$ for $2 \leq i \leq m$; and
\item[-]$y_i \in  V_1$ for $m+1 \leq i \leq n$. 
\end{itemize}
 Then for all $t  \in \CC$ the elements
$s_{y_i}(z_t)$, for $2 \leq i \leq n$,  are linearly independent.  
Continuing as in the proof of  Case 1 the result follows.
This finishes the proof of  Theorem~\ref{thm:extension}.

\subsection{A few more facts about $\tg\ms$} \label{sec:few-more-facts-tgms}
In this subsection we prove three more facts about $\Tad$-weights in $\tg\ms$. Let $\beta$ be such a weight, and let $\beta = \sum_{\lambda \in E}a_{\lambda}\lambda$ be the unique expression of $\beta$ as a $\Z$-linear combination of the elements of $E$. Proposition~\ref{prop:Tadweightnotantidom} guarantees that at least one of the $a_{\lambda}$ is positive. Proposition~\ref{prop:Kostant}, which is a consequence of a classical result attributed to Kostant, bounds the number and size of positive coefficients $a_{\lambda}$. Finally, Proposition~\ref{prop:simprootweight} gives a sufficient condition for $\beta$ to be a simple root and describes the $\Tad$-weight space of weight $\beta$ when the condition is met. The first two propositions do not use our extension criterion (Theorem~\ref{thm:extension}), while the third one does.

\begin{prop} \label{prop:Tadweightnotantidom}
Let $\beta$ be a $\Tad$-weight in $\Vgg$. Then there exists a simple
root $\alpha$ such that $\<\alpha^{\vee}, \beta\> >0$. Consequently, if $\beta = \sum_{\lambda \in E}a_{\lambda}\lambda$ is the unique expression of $\beta$ as a $\Z$-linear combination of the elements of $E$, then there exists $\lambda \in E$ with $\<\alpha^{\vee},\lambda\>>0$ and $a_{\lambda}>0$. 
\end{prop}
\begin{proof}
This follows by a standard argument from the fact, recalled in Proposition~\ref{prop:propsVggweights}, that $\beta$ is a nonzero element of $\<\sr\>_{\NN}$. For completeness, we include the details. Recall that we can equip the vector space $\Lambda_R \otimes_{\ZZ} \RR$, where $\Lambda_R = \<\sr\>_{\ZZ}$ is the root lattice, with a positive definite inner product $(\cdot \mid \cdot)$ such that for all $\alpha \in \sr$ and all $\gamma \in \Lambda_R$, we have that $\<\alpha^{\vee}, \gamma\>$ is a positive multiple of $(\alpha\mid\gamma)$, see e.g.~\cite[\S 18.3 and \S 18.4]{tauvel-yu-laag}. 

Since $\beta$ is an element of $\<\sr\>_{\NN}$ there exists, for every $\alpha \in \sr$, a nonnegative integer $n_{\alpha}$ such that $\beta = \sum_{\alpha \in \sr} n_{\alpha} \alpha$. Since $\beta \neq 0$, we know by the positive definiteness and the bilinearity of $(\cdot\mid\cdot)$ that  
\begin{equation}
0<(\beta\mid\beta) = \sum_{\alpha \in \sr} n_{\alpha} (\alpha\mid\beta). 
\end{equation}
It follows that there is some $\alpha \in \sr$ for which $(\alpha\mid\beta)>0$, whence $\<\alpha^{\vee},\beta\> >0$, which is what we had to prove.
\end{proof}

\begin{prop} \label{prop:Kostant}
Let $\beta$ be a $\Tad$-weight in $\tg\ms$. If $\beta = \sum_{\lambda \in E}a_{\lambda}\lambda$ is the unique expression of $\beta$ as a $\Z$-linear combination of the elements of $E$, then the sum of the elements of the set
\[\{a_{\lambda} \colon \lambda \in E \text{ and } a_{\lambda} > 0\}\]
is at most $2$. 
\end{prop}

\begin{proof}
This is a consequence of the following fact, attributed to Kostant (see, e.g., \cite[Proposition 28.6]{timashev-embbook}): the ideal $I=I(X_0)$ of the subvariety $X_0$ of $V$ is generated by the intersection, which we denote by $I_2$, of $I$ with the subspace $\C[V]_2$ of polynomials of degree $2$ in $\C[V]$. 
If we number the elements of $E$ as $\{\lambda_1, \lambda_2, \dots, \lambda_p\}$, then
\begin{equation} \label{eq:khoiu}
\C[V]_2 \cong [\oplus_{i=1}^p S^2 V(\lambda_i)] \oplus [\oplus_{1\leq i<j \leq p} V(\lambda_i) \otimes V(\lambda_j)]
\end{equation}  
as $G$-modules. 

It is shown in \cite{alexeev&brion-modaff} (and reviewed in \cite[Section 2.1]{degensphermod}) that because $E$ is linearly independent $\ms$ can be identified with an open subscheme of the invariant Hilbert scheme $\tHilb^G_{\wm}(V)$. It therefore follows from \cite[Proposition 1.13]{alexeev&brion-modaff} (and its proof) that we have natural isomorphisms
\begin{equation} \label{eq:iug}
\tg\ms \cong\Hom^G_{\C[X_0]}(I/I^2,\C[X_0]) \cong \Hom^G_{\C[V]}(I,\C[V]/I). 
\end{equation}
Recall from \cite[Section 2.2]{degensphermod} that, as reviewed at the start of Section~\ref{sec:crit-extens} above, the $\Tad$-action on $\tg\ms \cong \Hom^G_{\C[V]}(I,\C[V]/I)$ is induced by the action of $\GL(V)^G$ on $V$, using the homomorphism $f: T\to \GL(V)^G$ of \eqref{eq:homomf} on page~\pageref{eq:homomf}. The $\GL(V)^G$-action on $\Hom^G_{\C[V]}(I,\C[V]/I)$ induced by the $\GL(V)^G$-action on $V$ is given by
\[(t \cdot \rho)(h) = t\cdot \rho(t^{-1}\cdot h)\]
for $\rho \in  \Hom^G_{\C[V]}(I,\C[V]/I), t \in \GL(V)^G$ and $h \in I$.
Clearly, $\rho$ is completely determined by its restriction to $I_2$, since $I_2$ generates $I$ as an ideal. Being $G$-equivariant, $\rho$ sends each irreducible $G$-submodule $M$ of $I_2$ to $0$ or to a $G$-submodule of $\C[V]/I$ isomorphic to $M$. 

Since $\C[V]/I = \C[X_0]$, the $G$-module structure of this algebra is given by
\[\C[V]/I \cong \oplus_{\mu \in \wm} V(\mu) = \oplus_{(b_i) \in \NN^p} V(\sum_{i=1}^p b_i \lambda^*_i).\]
Moreover, as stated in the first paragraph of the proof of \cite[Proposition 28.6]{timashev-embbook}, 
$V(\sum_i b_i \lambda^*_i) \inn \C[V]/I$ is the image of $S^{b_1}V(\lambda_1^*) \otimes S^{b_2}V(\lambda_2^*) \otimes \ldots \otimes S^{b_p}V(\lambda_p^*) \inn \C[V]$ under the quotient map $\C[V] \onto C[V]/I$.  
It follows that $t=(t_1,t_2, \ldots, t_p)\in \GL(V)^G$ acts on $x \in V(\sum_i b_i \lambda^*_i) \inn  \C[V]/I$ by
\begin{equation} \label{eq:kjhkjh}
t \cdot x = t_1^{-b_1}t_2^{-b_2} \ldots t_p^{-b_p}x,
\end{equation}
since $t \mapsto t_1^{-b_1}t_2^{-b_2} \ldots t_p^{-b_p}$ is the character by which $\GL(V)^G$ acts on $S^{b_1}V(\lambda_1^*) \otimes S^{b_2}V(\lambda_2^*) \otimes \ldots \otimes S^{b_p}V(\lambda_p^*)$.

Now, suppose that $\rho$ is a $\Tad$-eigenvector in $\Hom^G_{\C[V]}(I,\C[V]/I)$ of weight $\beta$, and let $\beta = \sum_{i=1}^p a_i \lambda_i$ be the expression of $\beta$ as a $\Z$-linear combination of the elements of $E$.
The homomorphism $f:T \to \GL(V)^G$ in \eqref{eq:homomf} on page~\pageref{eq:homomf} relates the $\Tad$-action to the action of $\GL(V)^G$. Since $\beta \in \<E\>_{\Z}$ and $E$ is linearly independent there exists a unique character $\delta$ of $\GL(V)^G$ such that $\delta \circ f = \beta$. If $t=(t_1,\ldots,t_p) \in \GL(V)^G$, then $\delta(t) = t_1^{a_1}t_2^{a_2}\ldots t_p^{a_p}$. Moreover, $\delta$ is the $\GL(V)^G$-weight of $\rho$, which means that for $t=(t_1,t_2, \ldots, t_p)\in \GL(V)^G$, we have 
\begin{equation} \label{eq:ghdfh3}
t \cdot \rho = t_1^{a_1}t_2^{a_2} \ldots t_p^{a_p} \rho.
\end{equation}
Since $\rho \neq 0$ there exists an irreducible submodule $M$ of $I_2$ and an element $h$ of $M$ such that $\rho(h) \neq 0$. Because $I_2 \inn \C[V]_2$, it follows from the decomposition~(\ref{eq:khoiu}) that there exist $i,j \in \{1,2,\ldots p\}$, not necessarily distinct, such that for  $t=(t_1,t_2, \ldots, t_p)\in \GL(V)^G$ we have 
\begin{equation} \label{eq:kjhhk}
t \cdot h = t_i^{-1}t_j^{-1}h,
\end{equation}
since this is the action of $\GL(V)^G$ on $\C[V]_2$. 
Since $\rho(M) \neq 0$, there exists $(b_i) \in \NN^p$ such that $\rho(M) = V(\sum_{i=1}^p b_i \lambda^*_i) \inn \C[V]/I$. 

We then have for all $t=(t_1, \ldots, t_p) \in \GL(V)^G$ that
\begin{align}
(t \cdot \rho)(h) &= t \cdot \rho (t^{-1}\cdot h)\\
& = t_1^{-b_1}t_2^{-b_2} \ldots t_p^{-b_p}(\rho(t_i t_j h)) \\
& = t_i t_j t_1^{-b_1}t_2^{-b_2} \ldots t_p^{-b_p}(\rho(h)) \label{eq:890}
\end{align}
where the second equality uses equations (\ref{eq:kjhkjh}) and (\ref{eq:kjhhk}) and the third equality uses the $\C$-linearity of $\rho$. The proposition now follows from comparing (\ref{eq:890}) with equation (\ref{eq:ghdfh3}). 
\end{proof}

\begin{prop} \label{prop:simprootweight}
Let $\beta$ be a $\Tad$-weight in $\tg\ms$. 
If there exists $\lambda \in E$ satisfying the following two
properties:
\begin{itemize}
\item[-] $\lambda$ is of codimension one;
\item[-] $\lambda$ has a positive coefficient in the unique expression of
  $\beta$ as a linear combination of the elements of $E$;
\end{itemize}
then $\beta$ is a simple root. Moreover $\<\beta^{\vee}, \lambda\>
\neq 0$ and the $\Tad$-weight space in  $H^0(X_0, \shN_{X_0|V})^G
\isom T_{X_0}\tM^G_{\wm}$ of weight $\beta$ is spanned by the
section induced by $[X_{-\beta} v_{\lambda}] \in \Vgg$.  
\end{prop}
\begin{proof}
Since $\beta$ is a $\Tad$-weight in $\tg\ms \isom H^0(X_0, \shN_{X_0|V})^G$, there exists $v \in V$ of $\Tad$-weight $\beta$ such that $[v]$ is a nonzero element of $\Vgg$ and such that the corresponding section in $H^0(G\cdot x_0, \shN_{X_0|V})^G$ extends to $Z = G\cdot x_0 \cup G\cdot (x_0 - v_{\lambda})$. 
By Theorem~\ref{thm:extension}(\ref{item:extC}), we may assume that $v \in V(\lambda) \inn V$. To get a
contradiction, we assume that $\beta$
is not a simple root. Then, by Proposition~\ref{prop:propsVggweights}, there exists a simple root $\alpha$ such
that $\beta-\alpha$ is a positive root and 
\begin{equation}
0\neq X_{\alpha} \cdot v \in \<X_{-(\beta - \alpha)}
x_0\>_{\CC}. \label{eq:12}
\end{equation} 
Because $v \in V(\lambda)$, we have that
\begin{equation}X_{\alpha} \cdot v \in V(\lambda). \label{eq:12plus}
\end{equation} 
On the
other hand, because $\lambda$ is of codimension one, there exists
$\lambda' \in E \setminus \{\lambda\}$ such that $\<(\beta -
\alpha)^{\vee}, \lambda'\> \neq 0$. Consequently, the line $\<X_{-(\beta - \alpha)}
x_0\>_{\CC}$ has nonzero projection on $V(\lambda')$. We have
shown that \eqref{eq:12} and \eqref{eq:12plus} are in contradiction. That is, we
have shown that $\beta$ is a simple root.

By elementary highest weight theory, the $\Tad$-weight space of weight
$\beta$ in
$V(\lambda)$ is $\<X_{-\beta} v_{\lambda}\>_{\CC}$. This implies the
second assertion.  
\end{proof}

\section{Proof of Theorem~\ref{thm:cbc}} \label{sec:cases}

In this section, we prove Theorem~\ref{thm:cbc} through case-by-case verification: we verify that the theorem holds for each saturated indecomposable spherical module in  List~\ref{KnopL} below. The definition of `saturated' and `indecomposable' can be found in \cite[Section 5]{knop-rmks} or in \cite[Definition 4.1]{degensphermod}. The eight families (K1), (K2), (K3),
(K15), (K16), (K17), (K18) and (K21) were the subject of \cite[Section 5]{degensphermod}. 
Each subsection of this section corresponds to one of the remaining
families in the list: the proposition in each subsection asserts that Theorem~\ref{thm:cbc} holds for the family under consideration.

\begin{listabc}[Knop's List {\cite[Section 5]{knop-rmks}}] \label{KnopL}
The saturated indecomposable spherical modules $(\overline{G},W)$ are
\begin{itemize}
\item[(K1)] $(\GL(m)\times \GL(n), \CC^m \otimes \CC^n)$ with $1 \leq m\leq n$; 
\item[(K2)] $(\GL(n), \Sym^2\CC^n)$ with $1 \leq n$; 
\item[(K3)] $(\GL(n), \bigwedge^2\CC^n)$ with $2 \leq n$;  
\item[(K4)] $(\Sp(2n) \times \C^{\times}, \CC^{2n})$ with $1 \leq n$; 
\item[(K5)] $(\Sp(2n)\times \GL(2), \CC^{2n}\otimes \CC^2)$ with
  $2\leq n$;
\item[(K6)] $(\Sp(2n)\times \GL(3), \CC^{2n}\otimes \CC^3)$ with
  $3\leq n$;
\item[(K7)] $(\Sp(4)\times \GL(3), \CC^4 \otimes \CC^3)$;
\item[(K8)] $(\Sp(4) \times \GL(n), \CC^4 \otimes \CC^n)$ with $4
  \leq n$;  
\item[(K9)] $(\SO(n) \times \CC^{\times}, \CC^n)$ with $3 \leq n$;
\item[(K10)] $(\Spin(10) \times \CC^{\times}, \CC^{16})$;
\item[(K11)] $(\Spin(7) \times \CC^{\times}, \CC^{8})$;
\item[(K12)] $(\Spin(9) \times \CC^{\times}, \CC^{16})$;
\item[(K13)] $(\ssG_2 \times \CC^{\times}, \CC^7)$;
\item[(K14)] $(\ssE_6 \times \CC^{\times}, \CC^{27})$; 
\item[(K15)] $(\GL(n) \times \CC^{\times}, \bigwedge^2\CC^n \oplus \CC^n)$ with $4 \leq n$; 
\item[(K16)] $(\GL(n) \times \CC^{\times}, \bigwedge^2\CC^n \oplus (\CC^n)^*)$ with  $4 \leq n$; 
\item[(K17)] $(\GL(m) \times \GL(n), (\CC^m \otimes \CC^n) \oplus \CC^n)$ with $1\leq m, 2\leq n$; 
\item[(K18)] $(\GL(m) \times \GL(n), (\CC^m \otimes \CC^n) \oplus
  (\CC^n)^*)$ with $1\leq m, 2\leq n$;
\item[(K19)] $(\Sp(2n) \times \CC^{\times} \times \CC^{\times}, \CC^{2n} \oplus \CC^{2n})$
  with $2\leq n$;
\item[(K20)] $((\Sp(2n) \times \CC^{\times}) \times \GL(2), (\CC^{2n} \otimes
  \CC^2) \oplus \CC^2)$ with $2\leq n$;
\item[(K21)] $(\GL(m) \times \SL(2) \times \GL(n), (\CC^m \otimes \CC^2)\oplus (\CC^2 \otimes \CC^n))$ with 
$2\leq m \leq n$; 
\item[(K22)] $((\Sp(2m) \times \CC^{\times}) \times \SL(2) \times \GL(n),
  (\CC^{2m} \otimes \CC^2) \oplus (\CC^2 \otimes \CC^n))$ with $2 \leq
  m,n$;
\item[(K23)] $((\Sp(2m) \times \CC^{\times}) \times \SL(2) \times (\Sp(2n)
  \times \CC^{\times}),
  (\CC^{2m} \otimes \CC^2) \oplus (\CC^2 \otimes \CC^{2n}))$ with $2 \leq
  m,n$;
\item[(K24)] $(\Spin(8) \times \CC^{\times} \times \CC^{\times}, \CC^8_+ \oplus \CC^8_-)$.
\end{itemize}
\end{listabc}

\begin{remark}
The  indices $m$ and $n$ in family (K17) and family (K18) run through a larger set than that given in Knop's List in \cite{knop-rmks}. Knop communicated  the revised range of indices for these families to the second author. We remark that these cases do appear in the lists of \cite{leahy} and \cite{benson-ratcliff-mf}. In the family (K9) we suppose that $n \ge 3$, whereas in \cite{knop-rmks} it is required that $n\ge 2$. This correction to (K9) was already made in the revised version of \cite{knop-rmks} available on Knop's website.
\end{remark}

In the rest of the present section, we will use the same notations as in \cite[Section 5]{degensphermod}: in each subsection, $(\overline{G},W)$ will denote a member of the family from List~\ref{KnopL} under consideration.   
Given such a spherical module $(\overline{G},W)$,
\begin{itemize}
\item[-] $E$ denotes the basis of the weight monoid of $(\overline{G},W^*)$ (the elements of $E$ are called the `basic weights' in \cite{knop-rmks});
\item[-] $V= \oplus_{\lambda \in E} V(\lambda)$;
\item[-] $x_0 =  \sum_{\lambda \in E} v_{\lambda}$.
\end{itemize}
Except if stated otherwise, $G$ will denote  a connected subgroup of $\overline{G}$ containing $\overline{G}'$ such that $(G,W)$ is spherical. Recall that such a group $G$ is necessarily reductive.  To lighten notation, we will use $G'$ for the derived subgroup $\overline{G}'$ of $\overline{G}$. This should not cause confusion since $(\overline{G},\overline{G}) = (G,G) = G'$. We will use $p$ for the projection from the weight lattice of $G$ to the weight lattice of $G'$ (where we fix the maximal torus $T \cap G'$ of $G'$). We will use $\om, \om', \om''$ for weights of the first, second and third non-abelian factor of $G$, while $\varepsilon$ will refer to the character $\C^{\times} \to \C^{\times}, z \mapsto z$ of $\C^{\times}$.

\begin{remark} \label{rem:sphericalforsubgroup}
Let $(\overline{G},W)$ be a spherical module in Knop's List and let $G$ be a connected subgroup of $\overline{G}$ containing $\overline{G}'$. Theorem 5.1 in \cite{knop-rmks} gives a criterion which characterizes, in terms of the center of $G$, whether $(G,W)$ is a spherical module: $(G,W)$ is spherical if and only if the center of $G$ separates the weights in a certain subspace $\fa^* \cap \fz^*$ of the dual of the Lie algebra of the maximal torus of $\overline{G}$. The tables in \cite{knop-rmks}  give an explicit basis of $\fa^* \cap \fz^*$ for every spherical module $(\overline{G},W)$ in Knop's List. 
\end{remark}

\begin{remark}
\begin{enumerate}[(a)]
\item We recall from \cite[Remark 5.4]{degensphermod} that the $\Tad$-weight set we obtain below for each $\tg\ms^G$ is a basis of the monoid $-w_0\Sigma_W$, where $w_0$ is the longest element in the Weyl group of $G$ (instead of $-\Sigma_W$ as in Theorem~\ref{thm:main} where the $\Tad$-action from \cite{alexeev&brion-modaff} was used). 
\item As explained in \cite[Remark 5.6]{degensphermod}, the computations of the $\Tad$-weight sets of $\tg\ms^G$ we perform in this section confirm Knop's computation in \cite[Section 5]{knop-rmks} of the ``simple reflections'' of the little Weyl group of the
spherical modules under consideration.
\end{enumerate}
\end{remark}

\subsection{(K4) The modules $(\Sp(2n) \times\CC^{\times}, \CC^{2n})$ with $1 \leq n$} \label{subsec:case4}

Here 
\begin{align*}
&E =  \{\om_1 + \epsilon\};\\
&d_W = 0.
\end{align*}

We will make use of the following general lemma to treat this case, as well as the cases (K9), (K11) and (K13) below. 
\begin{lemma}\label{lem:dimx0}
Let $G$ be a connected reductive group and let $W$ be a spherical $G$-module. If $E^*$ is the basis of the weight monoid $\wm(W)$ of $W$ and $x_0 = \sum_{\lambda \in E} v_{\lambda} \in \oplus_{\lambda \in E} V(\lambda)$, then $\dim W = \dim \fg \cdot x_0$. 
\end{lemma}
\begin{proof}
This follows from \cite[Proposition 1.1]{jansou-ressayre} using the fact that $X_0 = \overline{G \cdot x_0}$ and $W$
have the same weight monoid.
\end{proof}

Applying this lemma to the modules $W$ in the family (K4) yields the following proposition.
\begin{prop}
The vector space $\Vg$ is zero-dimensional. 
In particular, $\dim \Vggp = d_W$. Consequently,  $\dim T_{X_0}\tM^G_{\wm} = d_W.$
\end{prop} 
\begin{proof}
Since $V\isom W$ and, by Lemma~\ref{lem:dimx0}, $\dim \fg \cdot x_0 = \dim W$, we have that $\Vg = \{0\}$. 
\end{proof}

\subsection{(K5) The modules $(\Sp(2n) \times \GL(2), \CC^{2n} \otimes \CC^2)$ with $2 \leq n$} \label{subsec:case5}

Here 
\begin{align*}
&E =  \{\om_1 + \om'_1, \om_2 + \om_2', \om_2'\};\\
&d_W = 2.
\end{align*}

\begin{prop}
The $\Tad$-module $\Vggp$ is multiplicity-free and has $\Tad$-weight set 
\[\{\alpha_1 + \alpha', \alpha_1 + 2\delta + \alpha_n\},\]
where $\delta = 0$ if $n=2$ and  $\delta = \alpha_2 + \alpha_3 + \ldots + \alpha_{n-1}$ if $n>2$. 
In particular, $\dim \Vggp = d_W$. Consequently,  $\dim T_{X_0}\tM^G_{\wm} = d_W$
\end{prop}

\begin{proof}
Note that $G' = \Sp(2n) \times \SL(2)$. Consider the $G'$-module $V' := V(\om_1 + \om'_1) \oplus V(\om_2)$ and its element $x_0' = v_{\om_1+\om'_1} + v_{\om_2}$. Observe that $G'_{x_0} = G'_{x'_0}$. 
Since $V(\om'_2)$ is one-dimensional, we have that
$\Vpggp \isom \Vggp$ as $\Tad$-modules. 

Recall that $p$ is the projection from the weight lattice $\wl$ of $G$ to the weight lattice of $G'$. The monoid $p(\<E\>_{\NN}) = \<\om_1+\om_1', \om_2\>_{\NN}$ is free and $G'$-saturated. By \cite[Theorems 3.1 and 3.10]{bravi&cupit}, $\Vpggp$ is multiplicity-free and its $\Tad$-weights belong to Table 1 in \cite[page 2810]{bravi&cupit}. By Proposition~\ref{prop:propsVggweights} the $\Tad$-weights of $\Vpggp$ also belong to $\Lambda_R \cap \<\om_1+\om_1', \om_2\>_{\ZZ}$.
A straightforward computation shows that
\begin{equation} \label{eq:3}
\Lambda_R \cap \<\om_1+\om_1', \om_2\>_{\ZZ} = 
\<\alpha_1 + 2\delta + \alpha_n, \alpha_1 + \alpha'\>_{\ZZ}
\end{equation}
Observe that the support of each of the two generators of $\Lambda_R \cap
\<\om_1+\om_1', \om_2\>_{\ZZ}$ in equation~(\ref{eq:3}) contains a simple
root not in the support of the other generator. Because the
$\Tad$-weights of $\Vpggp$ are linear combinations of the simple roots
with positive coefficents, it follows that they belong to $\<\alpha_1
+ 2\delta + \alpha_n, \alpha_1 + \alpha'\>_{\NN}$. One checks that for
$n>2$, the only $\Tad$-weights in \cite[Table 1]{bravi&cupit}
satisfying this requirement are $\alpha_1 + \alpha'$ and $\alpha_1 +
2\delta + \alpha_n$. 

For $n=2$ there is a third $\Tad$-weight that satisfies it, namely
$\gamma:=2\alpha_1 + 2\alpha_2$. The weight $\gamma$  occurs in $V'$,
with multiplicity one. More precisely, this $\Tad$-weight space is a
line $\CC v$ in $V(\omega_2) \inn V'$. We claim
that $[v] \in \frac{V'}{\fg' \cdot x'_0}$ does not belong to
$\Bigl(\frac{V'}{\fg' \cdot x'_0}\Bigr)^{G'_{x_0}}$. We prove the claim
by contradiction. Indeed, since
$\alpha_2$ is the only simple root such that $\gamma - \alpha_2 \in
R^+\cup \{0\}$, we have that  $0\neq [v] \in \Bigl(\frac{V'}{\fg' \cdot
  x_0'}\Bigr)^{G'_{x_0}}$ would, by Proposition~\ref{prop:propsVggweights}, imply
that $X_{\alpha_2}\cdot v$ is a nonzero element of $\<X_{-(\gamma -
  \alpha_2)}\cdot x_0'\>_{\C}$. This is absurd since $X_{-(\gamma -
  \alpha_2)}\cdot x_0'$ has nonzero projection onto the component
$V(\omega_1+\omega_1')$ of $V'$ and the claim is proved. 

We have shown that $\Vggp$ is multiplicity-free and that its
$\Tad$-weight set is a subset of the one in the statement of the
proposition. Since Proposition 2.5(1) of \cite{degensphermod} tells us that $\dim \tg\ms^G \ge d_W$ and Corollary 2.14 of \loccit\ says that $\tg\ms^G \inn \Vgg$, the
proposition follows. 
\end{proof}

\subsection{(K6) The modules $(\Sp(2n) \times \GL(3), \CC^{2n} \otimes \CC^3)$ with $3 \leq n$} \label{subsec:case6}

For these modules,
\begin{align*}
&E =  \{\lambda_1, \lambda_2, \lambda_3, \lambda_4, \lambda_5, \lambda_6\};\\
&d_W = 5,
\end{align*}
where
\begin{align*}\lambda_1& := \om_1 + \om_1';& \lambda_2&:=\om_2 + \om_2';& \lambda_3&:=\om_3+\om_3';\\ \lambda_4&:=\om_2';& \lambda_5&:=\om_1+\om_3';& \lambda_6&:=\om_2+\om_1'+\om_3'.
\end{align*}
We remark that $G = \overline{G}$ is the only connected group between $\overline{G}'$ and $\overline{G}$ for which these modules are spherical, cf.\ Remark~\ref{rem:sphericalforsubgroup}. Therefore, we assume that $G = \overline{G} = \Sp(2n) \times \GL(3)$ throughout this section.

In this section we will prove the following proposition.
\begin{prop} \label{prop:case6}
The $\Tad$-module  $T_{X_0}\tM^G_{\wm}$ is multiplicity-free. Its $\Tad$-weight set is
\begin{equation}
\{\alpha_1, \alpha_2, \alpha_2+\gamma, \alpha_1', \alpha_2'\}, \label{eq:case6weights}
\end{equation}
where $\gamma= \alpha_3$ if $n=3$ and $\gamma = 2(\alpha_3+\alpha_4+\ldots+\alpha_{n-1}) + \alpha_n$ if $n>3$. 
In particular, 
$\dim T_{X_0}\tM^G_{\wm} = d_W$.
\end{prop}
\begin{proof}
Let $\beta$ be a $\Tad$-weight in $\tg\ms^G$. 
It follows from Proposition~\ref{prop:Tadweightnotantidom} that at
least one $\lambda \in E$ has a positive coefficient in the
expression of $\beta$ as a $\ZZ$-linear combination of elements of
$E$. Note that all elements of $E$ except $\lambda_3$ have codimension
$1$. In particular, if $\lambda \neq \lambda_3$, then it follows
from Proposition~\ref{prop:simprootweight} that $\beta$ is a simple
root belonging to the set~\eqref{eq:case6weights}, and that its weight space has
dimension one. 

Since $\dim \tg\ms^G \ge d_W$ by \cite[Proposition 2.5(1)]{degensphermod}, what remains to prove the proposition is to show the following three claims:
\begin{description}
\item[Claim A] if $\lambda_3$ is the only element of $E$ which has a
positive coefficient in the expression of $\beta$, then $\beta= \alpha_2 + \gamma$ if $n>3$ and $\beta \in \{\alpha_2 + \gamma,2\alpha_2 + 2\gamma\}$ if $n=3$.  
\item[Claim B] suppose $n=3$; the $\Tad$-weight $\beta = 2\alpha_2 + 2\gamma$ does not occur in $\Vgg$, and therefore also not in the subspace $\tg\ms^G$.
\item[Claim C] the $\Tad$-weight $\beta= \alpha_2 + \gamma$ has multiplicity at most one in $\tg\ms^G$. 
\end{description}

We begin with Claim A. Recall from Proposition~\ref{prop:propsVggweights} that $\beta \in \<E\>_{\ZZ} \cap \<\Pi\>_{\NN}$. Straightforward computations show that
\begin{equation} \label{eq:case6lat}
\<E\>_{\ZZ} \cap \Lambda_R =
\<\alpha_1, \alpha_2, \alpha_1', \alpha_2',\gamma\>_{\ZZ}
\end{equation}
and that
\begin{align*}
\alpha_1 &= \lambda_1 + \lambda_5 - \lambda_6; \\
\alpha_2 &= \lambda_2 + \lambda_6 - \lambda_1 - \lambda_3 - \lambda_4;\\
\alpha'_1 &= \lambda_1 + \lambda_6 - \lambda_5 - \lambda_2;\\
\alpha'_2 &= \lambda_2 + \lambda_4 - \lambda_6;\\
\gamma &= 2\lambda_3 + \lambda_1 + \lambda_4 - \lambda_5 - \lambda_2 - \lambda_6.
\end{align*}
Let $K$ be the basis of $\<E\>_{\ZZ} \cap \Lambda_R$ given in equation~\eqref{eq:case6lat}. 
Since $\beta \in \<\sr\>_{\NN}$ and all elements of $K$ contain a simple root in their support, that is not in the support of any other element of $K$, it follows that $\beta \in \<K\>_{\NN}$. Therefore, there exist $A_1, A_2, \ldots, A_5 \in \NN$ such that 
\begin{equation}
\beta = A_1 \alpha_1 + A_2 \alpha_2 + A_3 \alpha'_1 + A_4 \alpha'_2 + A_5 \gamma. 
\end{equation}
From the hypothesis of Claim A, it follows that
\begin{equation}
\begin{cases}
A_1 - A_2 + A_3 + A_5\leq 0;\\
A_2-A_3+A_4- A_5 \leq 0;\\ 
 -A_2 + A_4 + A_5\leq 0; \\
 A_1 - A_3 - A_5\leq 0;\\ 
 -A_1 + A_2 + A_3 - A_4 - A_5\leq 0. \label{eq:c6ineqs}
\end{cases}
\end{equation}
Adding the first two inequalities in~\eqref{eq:c6ineqs} yields that $A_1 = A_4 = 0$. Then adding the first and the last gives that $A_3 = 0$. After substituting these values into the first and last inequalities, we deduce that $A_2 = A_5$. It follows that $\beta \in \<\alpha_2 + \gamma\>_{\NN}$. Using that $\beta$ is the sum of a simple root and an element of $\pr \cup \{0\}$ (see Proposition~\ref{prop:propsVggweights}) it follows that $\beta = \alpha_2 + \gamma$ if $n>3$ and that $\beta \in \{\alpha_2 + \gamma,2\alpha_2 + 2\gamma\}$ if $n=3$, This proves Claim A.

We proceed to Claim B. Let $n=3$ and fix $\beta=2\alpha_2 + 2\gamma = 2\alpha_2 + 2\alpha_3$. One deduces from the well-known decompositions into $T$-weight spaces of $V(\om_1)$, $V(\om_2)$ and $V(\om_3)$ that the $\Tad$-weight space in $V$ of $\Tad$-weight $\beta$ is a
line $\CC v$ in $V(\lambda_3) \inn V$. We prove Claim B
by contradiction. Indeed, since
$\alpha_3$ is the only simple root such that $\gamma - \alpha_3 \in
R^+\cup \{0\}$, we have that  $0\neq [v] \in \Vgg$ would, by Proposition~\ref{prop:propsVggweights}, imply
that $X_{\alpha_3}\cdot v$ is a nonzero element of $\<X_{-(\gamma -
  \alpha_3)}\cdot x_0\>_{\CC}$. This is absurd since $X_{-(\gamma -
  \alpha_3)}\cdot x_0$ has nonzero projection onto the components $V(\lambda_2)$ and $V(\lambda_6)$ of $V$. Claim B is proved.

Finally, we show Claim C. We fix $\beta = \alpha_2 + \gamma$. We will show that the $\Tad$-weight $\beta$ has multiplicity at most one in $\Vgg$. Since $\tg\ms^G \inn \Vgg$, this implies Claim C. First off, we claim that the $\Tad$-weight $\beta$ only occurs in $V(\lambda_2)$, $V(\lambda_3)$ and in $V(\lambda_6)$. Indeed, $\beta$ belongs to the root lattice of $\Sp(2n)$, and does not occur as a $\Tad$-weight in $V(\om_1)$. Let $Z$ be the subspace of $V(\lambda_2) \oplus V(\lambda_3) \oplus V(\lambda_6)$ consisting of $\Tad$-eigenvectors $v$ of $\Tad$-weight $\beta$ that satisfy the following three conditions:
\begin{align}
&X_{\alpha_2} \cdot v \in \<X_{-(\beta - \alpha_2)}x_0\>_{\CC}; \label{eq:7c1} \\
&X_{\alpha_3} \cdot v \in \<X_{-(\beta - \alpha_3)}x_0\>_{\CC}; \text{ and}\label{eq:7c2} \\
&X_{\alpha_k} \cdot v = 0 \text{ for all $k \in \{1,2,\ldots,n\} \setminus \{2,3\}$}. \label{eq:7c3}
\end{align}
Since $\alpha=\alpha_2$ and $\alpha=\alpha_3$ are the only simple roots such that $\beta - \alpha \in R^+\cup\{0\}$ it follows from Proposition~\ref{prop:propsVggweights} that every $v \in V^{\beta}$ such that $0\neq [v] \in \Vgg$ satisfies \eqref{eq:7c1}, \eqref{eq:7c2} and \eqref{eq:7c3}.
To show Claim C it is therefore enough to prove that 
\begin{equation}
\dim Z \leq 2, \label{eq:dimZleq2}
\end{equation}
since the nonzero vector $X_{-\beta}\cdot x_0$ belongs to $\fg\cdot x_0 \cap Z$. 

\label{sp2ndata}
To prove the inequality (\ref{eq:dimZleq2}) we will make use of the explicit description of $\sp(2n)$ and its root operators given in the proof of \cite[Theorem 2.4.1]{goodman-wallach-gtm}, as well as the notations therein. In particular, we have a basis $\{e_1,e_2,\ldots,e_n, e_{-1},e_{-2}, \ldots, e_{-n}\}$ of $\CC^{2n}$ and a $\ZZ$-basis $\{\eps_1, \eps_2, \ldots, \eps_n\}$ of the weight lattice of $\Sp(2n)$ such that $e_i$ has weight $\eps_i$ and $e_{-k}$ has weight $-\eps_k$ in the defining representation of $\Sp(2n)$ on $\CC^{2n}$. In terms of the basis $\{\eps_1, \eps_2, \ldots, \eps_n\}$ of the weight lattice, the simple roots of $\Sp(2n)$ are $\alpha_i = \eps_i - \eps_{i+1}$ for $i \in \{1,2,\ldots,n-1\}$ and $\alpha_{n} = 2\eps_n$. Moreover, for each root $\delta$ we have a root operator $X_{\delta} \in \sp(2n)^{\delta}$. In view of conditions (\ref{eq:7c1}), (\ref{eq:7c2}) and (\ref{eq:7c3}) we will make use of the root operators associated to the simple roots and to the negative roots $-(\beta-\alpha_2)= -2\epsilon_3$ and $-(\beta - \alpha_3)=-\epsilon_2 - \epsilon_4$. The action of these operators on the given basis of the defining representation $\CC^{2n}$ of $\Sp(2n)$ is as follows:
\begin{align}
X_{\alpha_i} \cdot e_k &= \begin{cases}
e_i &\text{if $k = i+1$}; \\
-e_{-(i+1)} &\text{if $k = -i$};\\
0 &\text{if $k \notin \{i,-i\}$} 
\end{cases} &&\text{for $i \in \{1,2,\ldots,n-1\}$;} \\
X_{\alpha_n} \cdot e_k  &= \begin{cases}
e_n &\text{if $k=-n$};\\
0 &\text{if $k \neq -n$};
\end{cases} \\
X_{-2\epsilon_i}\cdot e_k&=
\begin{cases}
e_{-i} &\text{if $k = i$};\\
0 &\text{if $k \neq i$}
\end{cases} &&\text{for $i \in \{1,2,\ldots,n\}$;}\\
X_{-\epsilon_i - \epsilon_j}\cdot e_k&=
\begin{cases}
e_{-j} &\text{if $k = i$};\\
e_{-i} &\text{if $k = j$};\\
0 &\text{if $k \notin \{i,j\}$}
\end{cases}&&\text{where $1\leq i < j \leq n$;}
\end{align}

Note that 
$$
\beta = \alpha_2 + \gamma = -\om_1 + \om_3 = \eps_2 + \eps_3.
$$
We now identify the weight spaces of this $\Tad$-weight in the representations $V(\om_2)$ and $V(\om_3)$ of $\Sp(2n)$. 
A vector in $V(\om_2)$ has $\Tad$-weight $\beta$ if and only if it has $T$-weight $\om_2 - \beta = \eps_1 - \eps_3$. We identify $V(\om_2)$ with the sub-$\Sp(2n)$-representation of $\wedge^2 \CC^{2n}$ with highest weight vector $e_1 \wedge e_2$. Then the $T$-weight space in $V(\om_2)$ of weight $\eps_1 - \eps_3$ is the line spanned by
\begin{equation} \label{eq:Tweightspom2}
e_1 \wedge e_{-3}.
\end{equation}  
A vector in $V(\om_3)$ has $\Tad$-weight $\beta$ if and only if it has $T$-weight $\om_3 - \beta = \eps_1$. As is well-known, $V(\om_3)$ is the irreducible component of the $\Sp(2n)$-module $\wedge^3 \CC^{2n}$ generated by the highest weight vector $e_1\wedge e_2 \wedge e_3$. In the larger module $\wedge^3 \CC^{2n}$ the $T$-weight space of weight $\eps_1$ is spanned by the following vectors
\begin{equation} \label{eq:Tweightspom3}
e_1 \wedge e_2 \wedge e_{-2}, e_1 \wedge e_3 \wedge e_{-3}, \ldots , e_1 \wedge e_n \wedge e_{-n}.
\end{equation} 

It follows from the previous paragraph that if $v\in Z$, then there exist $A_1, A_2 \in \CC$ and  $B_2, B_3, \ldots , B_n \in \CC$ such that
\begin{equation}
v = A_1(e_1 \wedge e_{-3} \otimes v_{\om_2'}) + A_2(e_1 \wedge e_{-3} \otimes v_{\om_1' + \om_3'})  
+ \sum_{k=2}^n B_k (e_1 \wedge e_k \wedge e_{-k} \otimes v_{\om_3'}). 
\end{equation}   
Straightforward computations using the root operators show that conditions (\ref{eq:7c1}), (\ref{eq:7c2}) and (\ref{eq:7c3}) imply that 
\begin{align}
&A_1 = A_2 = -B_3 + B_4; \label{eq:cvinZ1}\\
&B_4 = B_5 = \ldots = B_n.\label{eq:cvinZ2}
\end{align}
This implies that $\dim Z\leq 3$. Note that the vector $v_1 \otimes v_{\om_3'} \in \wedge^3\CC^{2n} \otimes V(\om_3')$, where 
$$v_1 = \sum_{k=2}^ne_1 \wedge e_k \wedge e_{-k},$$ satisfies the equations (\ref{eq:cvinZ1}) and (\ref{eq:cvinZ2}). It is straightforward to check that $v_1$ is a highest weight vector. It follows that  $v_1$ is an element of the $\Sp(2n)$-stable complement to $V(\om_3)$ in $\wedge^3 \CC^{2n}$. Consequently, the line spanned by $v_1 \otimes v_{\om_3'}$ is not contained in $Z$, and $\dim Z \leq 2$. This proves Claim C, and the proposition.
\end{proof}

\subsection{(K7) The module $(\Sp(4) \times \GL(3), \CC^4 \otimes \CC^3)$} \label{subsec:case7}
For this module we have
\begin{align*}
E &= \{\om_1 + \om_1', \om_2 + \om_2', \om_2', \om_1 + \om_3', \om_2 + \om_1' + \om_3'\} \\
d_W &= 4.
\end{align*}

\begin{prop} \label{prop:case7}
The $\Tad$-module $T_{X_0}\tM^G_{\wm}$ is multiplicity-free. Its
$\Tad$-weight set is
\begin{equation} \label{eq:case7wts}
\{\alpha_1, \alpha_2, \alpha_1', \alpha_2'\}.
\end{equation}
In particular, $\dim T_{X_0}\tM^G_{\wm} = d_W$.
\end{prop}
\begin{proof}
Let $\beta$ be a $\Tad$-weight in  $\tg\ms^G$.
It follows from Proposition~\ref{prop:Tadweightnotantidom} that at
least one $\lambda \in E$ has a positive coefficient in the
expression of $\beta$ as a $\ZZ$-linear combination of elements of
$E$. Note that all elements of $E$ have codimension
$1$. It follows
from Proposition~\ref{prop:simprootweight} that $\beta$ is a simple
root and that its weight space has
dimension one. Since the set (\ref{eq:case7wts}) contains all simple roots of $G$, we can conclude that $\beta$ belongs to this set and that $\dim \tg\ms^G \leq d_W$. The proposition now follows from the a priori estimate $\dim \tg\ms^G \ge d_W$, see \cite[Proposition 2.5(1)]{degensphermod}.  
\end{proof}

\subsection{(K8) The modules $(\Sp(4) \times \GL(n), \CC^4 \otimes \CC^n)$
  with $4\leq n$} \label{subsec:case8}
We put $\lambda_1 = \om_1 + \om_1', \lambda_2 = \om_2 + \om_2',
\lambda_3 = \om_2', \lambda_4 = \om_1 + \om_3', \lambda_5 = \om_2 +
\om_1' + \om_3', \lambda_6 = \om_4'$.
Then
\begin{align*}
&E =  \{\lambda_1, \lambda_2, \lambda_3, \lambda_4, \lambda_5, \lambda_6\};\\
&d_W = 5.
\end{align*}

\begin{prop} \label{prop:case8}
The $\Tad$-module $T_{X_0}\tM^G_{\wm}$ is multiplicity-free. Its
$\Tad$-weight set is
\begin{equation} \label{eq:13}
\{\alpha_1, \alpha_2, \alpha_1', \alpha_2', \alpha_3'\}.
\end{equation}
In particular, $\dim T_{X_0}\tM^G_{\wm} = d_W$.
\end{prop}
\begin{proof}
Let $\beta$ be a $\Tad$-weight in  $\tg\ms^G$.
It follows from Proposition~\ref{prop:Tadweightnotantidom} that at
least one $\lambda \in E$ has a positive coefficient in the
expression of $\beta$ as a $\ZZ$-linear combination of elements of
$E$. Note that all elements of $E$ except $\lambda_6$ have codimension
$1$. In particular, if $\lambda \neq \lambda_6$, then it follows
from Proposition~\ref{prop:simprootweight} that $\beta$ is a simple
root belonging to the set~\eqref{eq:13}, and that its weight space has
dimension one. 

Consequently, to prove the proposition, what remains is to show
that $\lambda_6$ cannot be the only element of $E$ which has a
positive coefficient in the expression of $\beta$ as a linear
combination of the elements of $E$. 

Note that $\overline{G}' = \Sp(4)\times \SL(n)$ and therefore that $\dim \overline{G}/\overline{G}' = 1$.
For $n=4$ the only connected group between $\overline{G}'$ and $\overline{G}$, for which $W$ is spherical, is $G=\overline{G}$. On the other hand, if $n>4$, then there are two such groups: $G=\overline{G}$ and $G=\overline{G}'$; cf.\ Remark~\ref{rem:sphericalforsubgroup}.
Straightforward computations show that 
\begin{equation} \label{eq:case8lat}
\<E\>_{\ZZ} \cap \Lambda_R =
\begin{cases} 
\<\alpha_1, \alpha_2, \alpha_1', \alpha_2',\alpha_3'\>_{\ZZ}
& \text{if $G = \overline{G}$ or $n$ is odd or $n=4$};\\
\<\alpha_1, \alpha_2, \alpha_1', \alpha_2',\alpha_3', \gamma\>_{\ZZ}
&\text{if $G = \overline{G}'$ and $n$ is even and $n>4$}
\end{cases}
\end{equation}
where 
\begin{equation} \label{eq:case8gamma}
\gamma= (n-4)\alpha_4'+(n-5)\alpha_5' + \ldots +2\alpha'_{n-2}
+ \alpha'_{n-1}. 
\end{equation}
Let $K$ be the basis of $\<E\>_{\ZZ} \cap \Lambda_R$ in equation~(\ref{eq:case8lat}). 
Since $\beta \in \<\sr\>_{\NN}$ and all elements of $K$ contain a simple root in their support, which is not in the support of any other element of $K$, it follows that $\beta \in \<K\>_{\NN}$. 
We have the following equalities:
\begin{align*}
\alpha_1 &= \lambda_1+\lambda_4-\lambda_5;\\
\alpha_2 & = \lambda_2 + \lambda_5 - \lambda_3-\lambda_1-\lambda_4; \\
\alpha'_1 &= \lambda_1 + \lambda_5 -\lambda_4 - \lambda_2; \\
\alpha'_2 &= \lambda_2 + \lambda_3 - \lambda_5; \\
\alpha'_3 & = \lambda_4+ \lambda_5-\lambda_6 - \lambda_1-\lambda_2; \\
\gamma &= (n-3)\lambda_6 - \frac{n-4}{2}[\lambda_4 + \lambda_5 - \lambda_1 - \lambda_2 + \lambda_3].
\end{align*}
As one easily sees, $\gamma$ is the only element of $K$ in which $\lambda_6$ has a positive coefficient. Consequently, the Proposition follows from equation~(\ref{eq:case8lat}) for $G=\overline{G}$, for odd $n$ and for $n=4$.

We now assume that $G=\overline{G}'$, that $n$ is even and at least $6$ and that $\lambda_6$ has a positive coefficient in $\beta$. We will come to a contradiction. Our assumptions imply that $\gamma$ has a positive coefficient in the expression of $\beta$ as an $\NN$-linear combination of the elements of $K$. Recall from Proposition~\ref{prop:propsVggweights} that $\beta$ is the sum of a simple root and an element of $\pr \cup \{0\}$. By equation (\ref{eq:case8gamma}) this is only possible if $n=6$ and if $\beta$ is one of the following three elements of $\<K\>_{\NN}$:
\begin{align*}
\beta_1 &:= \alpha'_1 + \alpha'_2 + \alpha'_3 + \gamma;\\
\beta_2 &:= \alpha'_2 + \alpha'_3 + \gamma; \\
\beta_3 &:= \alpha'_3 + \gamma.
\end{align*}
Since $\beta_1 = \lambda_1 + 2\lambda_6 - \lambda_4$ and $\beta_2 = \lambda_2 + 2\lambda_6 -\lambda_5$, it follows that $\beta$ cannot be either of them by Proposition~\ref{prop:Kostant}. 
Finally, $\beta = \beta_3$ is not possible because  if $v \in V^{\beta_3}$ with $0\neq [v] \in \Vgg$, then it follows from Proposition~\ref{prop:propsVggweights} that $X_{\alpha'_4}\cdot v$ is a nonzero element in $\<X_{-(\beta_3 - \alpha'_4)} \cdot x_0\>_{\C}$, since $\alpha'_4$ is the only simple root such that $\beta_3-\alpha'_4 \in R^+\cup\{0\}$.  Since  $X_{-(\beta_3 - \alpha'_4)} \cdot x_0$ has nonzero projection on $V(\lambda_4)$, so does $v$, but $\beta_3$ does not occur as a $\Tad$-weight in $V(\lambda_4)$ as follows immediately from the well known list of $T$-weights in $V(\omega_3')$. This completes the proof.
\end{proof}

\subsection{(K9) The modules $(\SO(n) \times \CC^{\times}, \CC^n)$ with $3
  \leq n$} \label{subsec:case9}

For these modules
\begin{align*}
&E =  \{\om_1 + \epsilon, 2\epsilon\};\\
&d_W = 1.
\end{align*}

\begin{prop} \label{prop:case9}
The $\Tad$-module $\Vggp$ is one-dimensional. Its weight is
\begin{align*}
&&2\alpha_1 + 2\alpha_2 + \ldots + 2\alpha_{(n/2)-2} + \alpha_{(n/2)-1}
+ \alpha_{n/2} &&\text{if $n$ is even};\\
&&2\alpha_1 + 2\alpha_2 + \ldots + 2\alpha_{(n-1)/2} &&\text{if $n$ is odd}.
\end{align*}
In particular, $\dim \Vggp = d_W$. Consequently,  $\dim T_{X_0}\tM^G_{\wm} = d_W$.
\end{prop}
\begin{proof}
Observe that $\dim V = \dim W + 1$, since $V(2 \epsilon)$ is
one-dimensional. Since $\dim W = \dim \fg \cdot x_0$ (by Lemma~\ref{lem:dimx0}), this implies
that $\dim \Vg =1$. Since $d_W=1$, this implies that $\dim \Vggp =
1$ by \cite[Proposition 2.5(1) and Corollary 2.14]{degensphermod}. 

We now find the $\Tad$-weight of $\Vggp$. Using the well-known
$T$-weight space decomposition of the $G$-module $V(\om_1 + \epsilon) \isom \CC^{n}
\otimes \CC_{\epsilon}$, the fact that $V(2\epsilon) \subset \fg \cdot
x_0$ and \cite[Lemma 2.16(4)]{degensphermod}, one readily checks that
the one-dimensional $\Tad$-module $\Vg$ has the $\Tad$-weight given in
the proposition.
\end{proof}

\subsection{(K10) The module $(\Spin(10)\times \CC^{\times}, \CC^{16})$} \label{subsec:case10}
Here
\begin{align*}
&E =  \{\om_5 + \epsilon, \om_1 + 2\epsilon\};\\
&d_W = 1.
\end{align*}

\begin{prop} \label{prop:case10}
The $\Tad$-module $\Vggp$ is one-dimensional. Its weight is
\begin{align*}
\alpha_2+2\alpha_3 + \alpha_4 + 2\alpha_5.
\end{align*}
In particular, $\dim \Vggp = d_W$. Consequently,  $\dim T_{X_0}\tM^G_{\wm} = d_W.$ 
\end{prop}

\begin{proof}
Recall that $p$ is the projection from the weight lattice $\wl$ of $G$ to the weight lattice of $G'=\Spin(10)$.
We first observe that $W = V(\om_5+\epsilon)$ is spherical for $G' =
\Spin(10)$ (cf.\ Remark~\ref{rem:sphericalforsubgroup}) and that its weight monoid $p(\wm)$ is
$G'$-saturated. By \cite[Corollary 2.27]{degensphermod} it follows
that $\dim \Vggp = d_W$. 

While we do not need it for Theorem~\ref{thm:cbc}, we give a proof of the claim that the $\Tad$-weight of $\Vggp$ is $\alpha_2+2\alpha_3 +
\alpha_4 + 2\alpha_5$. A straightforward calculation shows that
\begin{align}
p(\<E\>_{\ZZ}) \cap \Lambda_R &= \<2 \alpha_1 + \alpha_2 - \alpha_5, -2\alpha_1 + 2\alpha_3 + \alpha_4 + 3 \alpha_5\>_{\ZZ}\\
 &= \<\beta_1, \beta_2\>_{\ZZ}. \label{eq:case10rl}
\end{align}
where $\beta_1 = \alpha_2 + 2\alpha_3 + \alpha_4 + 2\alpha_5$ and $\beta_2 =  2\alpha_1 + 2\alpha_2 + 2\alpha_3 + \alpha_4 + \alpha_5$. 
If $\beta$ is a $\Tad$-weight occurring in $\Vggp$, then, by Proposition~\ref{prop:propsVggweights}, 
\begin{eqnarray}
\beta &\in & \<\sr\>_{\NN} \cap p(\<E\>_{\ZZ}) = \<\sr\>_{\NN} \cap \<\beta_1, \beta_2\>_{\ZZ}; \text{ and} \label{eq:c10_1} \\ 
\beta &\in & \sr + (R^{+} \cup\{0\}) \label{eq:c10_2}
\end{eqnarray}
In the root system of type $\ssD_5$, if an element of $\sr + (R^{+} \cup\{0\})$ is written as a linear combination of the simple roots, then none of the coefficients are greater than $3$. Consequently, \eqref{eq:c10_1} and \eqref{eq:c10_2} imply that there exists $a,b \in \Z$ such that $\beta = a\beta_1 + b\beta_2$ and
\begin{align}
3 \geq  2b  &\geq   0   \label{eq:c10_3} \\
3 \geq  2a+2b  &\geq  0  \label{eq:c10_4} \\
3 \geq 2a+b &\geq 0 \label{eq:c10_5}
\end{align}
It follows from \eqref{eq:c10_3} that $b \in \{0,1\}$. If $b=0$, then it follows from \eqref{eq:c10_4} that $a \in \{0,1\}$. If $b=1$, then it follows from \eqref{eq:c10_5} that $a\in \{0,1\}$, and then \eqref{eq:c10_4} implies that $a=0$. Since $\beta \neq 0$, we have shown that $\beta=\beta_1$ or $\beta=\beta_2$. 

To finish the proof, we have to show that $\beta_2$ cannot occur as a $\Tad$-weight in $\Vggp$. To get a contradiction, suppose that $v \in V^{\beta_2}$ such that $0 \neq [v] \in \Vggp$. Since $\alpha=\alpha_1$ is the only simple root such that $\beta_2 -\alpha \in R^+\cup \{0\}$, it follows from Proposition~\ref{prop:propsVggweights} that 
\begin{equation}
X_{\alpha_1} \cdot v \in \<X_{-(\beta_2-\alpha_1)}\cdot x_0\>_{\C} \setminus\{0\}\label{eq:c10_6}
\end{equation}
Since $\<(\beta_2-\alpha_1)^{\vee},\om_1\> \neq 0$ and $\<(\beta_2-\alpha_1)^{\vee},\om_5\> \neq 0$, the vector $X_{-(\beta_2-\alpha_1)}\cdot x_0$ has nonzero projection on both irreducible components of $V$.
This is in contradiction with \eqref{eq:c10_6}, since the $\Tad$-weight $\beta$ does not occur in $V(\om_5)$. 
This finishes the proof. 
\end{proof}

\begin{remark}
The fact that the $\Tad$-weight of $\Vggp$ is $\alpha_2+2\alpha_3 +
\alpha_4 + 2\alpha_5$, which is equal to $\epsilon_2 + \epsilon_3 +
\epsilon_4 + \epsilon_5$, can also be deduced from the description of the little
Weyl group of the module $W^*$ given in \cite{knop-rmks} (see~\cite[Remark 2.8]{degensphermod} for some context.)
\end{remark}

\subsection{(K11) The module $(\Spin(7)\times \CC^{\times}, \CC^{8})$} \label{subsec:case11}
Here
\begin{align*}
&E =  \{\om_3 + \epsilon, 2\epsilon\};\\
&d_W = 1.
\end{align*}

\begin{prop} \label{prop:case11}
The $\Tad$-module $\Vggp$ is one-dimensional. Its weight is
\begin{align*}
\alpha_1+2\alpha_2+3\alpha_3.
\end{align*}
In particular, $\dim \Vggp = d_W$. Consequently,  $\dim T_{X_0}\tM^G_{\wm} = d_W$
\end{prop}

\begin{proof}
The proof that $\dim \Vggp = 1$ is exactly like in the proof of Proposition~\ref{prop:case9}.
We now find the $\Tad$-weight of $\Vggp$, also like in the proof of
Proposition~\ref{prop:case9}.  Using the $T$-weight space decomposition of
the $G$-module $V(\om_3 + \epsilon)$ (which can be computed by hand or
with \emph{LiE}~\cite{LiE1992}), the fact that $V(2\epsilon) \subset \fg \cdot
x_0$ and \cite[Lemma 2.16(4)]{degensphermod}, one readily checks that
the one-dimensional $\Tad$-module $\Vg$ has the $\Tad$-weight given in
the proposition.
\end{proof}

\subsection{(K12) The module $(\Spin(9) \times \CC^{\times}, \CC^{16})$} \label{subsec:case12}
Here 
\begin{align*}
&E =  \{\om_4 + \epsilon, \om_1 + 2\epsilon, 2\epsilon\};\\
&d_W = 2.
\end{align*}

\begin{prop}  \label{prop:case12}
The $\Tad$-module $\Vggp$ is multiplicity-free and has $\Tad$-weight set 
\[\{\alpha_1 + \alpha_2 + \alpha_3 + \alpha_4, \alpha_2 + 2\alpha_3 + 3\alpha_4\}.\]
In particular, $\dim \Vggp = d_W$. Consequently,  $\dim T_{X_0}\tM^G_{\wm} = d_W$
\end{prop}
\begin{proof}
Observe that $G' = \Spin(9)$. Consider the $G'$-module $V':= V(\om_1) \oplus V(\om_4)$ and its element $x_0' = v_{\om_1} + v_{\om_4}$. Observe that $G'_{x_0}=G'_{x_0'}$. Since $V(2 \epsilon)$ is one-dimensional, we have that
$\Vpggp \isom \Vggp$ as $\Tad$-modules.  Put $E' = p(E) = \{\om_1, \om_4\}$. 

Let $\beta$ be a $\Tad$-weight of $\Vpggp$. Then, by Proposition~\ref{prop:propsVggweights}, 
\begin{align}
\beta & \in \<E'\>_{\ZZ} \cap \<\sr\>_{\NN} \label{eq:c12_1}\\
\beta &\in \sr +( R^+ \cup \{0\}) \label{eq:c12_2}
\end{align}
A straightforward computation shows that 
\begin{equation}
\<E'\>_{\ZZ} \cap \<\sr\>_{\Z} = \<\beta_1, \beta_2\>_{\Z}, \label{eq:c12_3}
\end{equation}
where $\beta_1 = \alpha_1 + \alpha_2 +\alpha_3 + \alpha_4$ and $\beta_2 = \alpha_2 + 2\alpha_3 + 3\alpha_4$. 
The explicit description of $R^+$ for
$\Spin(9)$ (see, e.g. \cite[Planche II]{bourbaki-geadl47}) shows that \eqref{eq:c12_2} implies that if 
$\beta$ is written as a linear combination of the simple roots, then the coefficient of $\alpha_1$ is at most $2$ and that of the other simple roots is at most $3$. Combined with \eqref{eq:c12_1} and \eqref{eq:c12_3} this implies that there exist $a,b \in \Z$ such that $\beta = a\beta_1 + b\beta_2$ and
\begin{align*}
2\geq a &\geq 0; \text{ and}\\
3\geq a+3b &\geq 0
\end{align*}
This system implies that $(a,b) \in \{(0,0),(1,0), (2,0), (0,1)\}$. 
Since $\beta \neq 0$, we have shown that
\begin{equation} \label{eq:2}
\beta \in \{\alpha_1 + \alpha_2 +
\alpha_3 + \alpha_4, 2(\alpha_1 + \alpha_2 + \alpha_3 + \alpha_4),
\alpha_2 + 2\alpha_3 + 3\alpha_4\}.
\end{equation}

We claim that $\beta_3:=2(\alpha_1 + \alpha_2 + \alpha_3 + \alpha_4)$ does not
occur as a $\Tad$-weight in $\Vpggp$. We will argue by
contradiction; assume $v \in V'$ is a $\Tad$-eigenvector of weight
$\beta_3$ such that $[v]$ is nonzero in $\Vpggp$. Using the explicit description
of $R^+$ once more, one readily checks that $\alpha=\alpha_1$ is the
only simple root $\alpha$ such that $\beta_3- \alpha \in R^{+}\cup
\{0\}$. By Proposition~\ref{prop:propsVggweights} this implies that
\begin{equation}\label{eq:1}
X_{\alpha_1} \cdot v \in \<X_{-(\beta_3 - \alpha_1)}\cdot x'_0\>_{\CC} \setminus\{0\}.
\end{equation} 
Because $\<(\beta_3 - \alpha_1)^{\vee}, \om_1\> \neq 0$ and
$\<(\beta_3 - \alpha_1)^{\vee}, \om_4\> \neq 0$, we have that
$X_{-(\beta_3 - \alpha_1)} x'_0$ has nonzero projection to both
summands $V(\om_1)$ and $V(\om_4)$ of $V'$. On the other hand, the
following computation in \emph{LiE} shows that $\beta_3$ does not
occur as a $\Tad$-weight in $V(\om_4)$. 
\begin{quote}
\begin{verbatim}
setdefault(B4)
omega4=[0,0,0,1]
beta3=[2,0,0,0]
Demazure(omega4)|(omega4-beta3)
-- output: 0
\end{verbatim}
\end{quote}
This implies that $v$ is in
the kernel of the projection onto the summand $V(\om_4)$ of $V'$,
which is in contradiction with equation~(\ref{eq:1}). This proves the
claim. 

Since the monoid $\<E'\>_{\NN}$ is free and $G'$-saturated, we know
that the $\Tad$-module $\Vpggp$ is multiplicity-free by \cite[Theorem
3.10]{bravi&cupit}. Equation~(\ref{eq:2}) and the claim above then
imply that $\dim \Vpggp \leq 2$. Since $d_W=2$ this proves the
proposition, because $\tg\ms^G \inn \Vgg \inn \Vpggp$ and $\dim \tg\ms^G \ge d_W$, by \cite[Proposition 2.5(1)]{degensphermod}. 
\end{proof}

\subsection{(K13) The module $(\ssG_2 \times \CC^{\times}, \CC^7)$} \label{subsec:case13}
Here
\begin{align*}
&E =  \{\om_1 + \epsilon, 2\epsilon\};\\
&d_W = 1.
\end{align*}

\begin{prop} \label{prop:case13}
The $\Tad$-module $\Vggp$ is one-dimensional. Its weight is
\begin{align*}
4\alpha_1 + 2\alpha_2.
\end{align*}
In particular, $\dim \Vggp = d_W$. Consequently,  $\dim T_{X_0}\tM^G_{\wm} = d_W$
\end{prop}
\begin{proof}
Same argument as for Proposition~\ref{prop:case11}.
\end{proof}

\subsection{(K14) The module $(\ssE_6 \times \CC^{\times}, \CC^{27})$} \label{subsec:case14}
Here 
\begin{align*}
&E =  \{\om_1 + \epsilon, \om_6 + 2\epsilon, 3 \epsilon\};\\
&d_W = 2.
\end{align*}

\begin{prop}
The $\Tad$-module $\Vggp$ is multiplicity-free and has $\Tad$-weight set 
\[\{\alpha_2 +
\alpha_3 + 2\alpha_4 + 2\alpha_5 + 2\alpha_6, 2\alpha_1
+ \alpha_2 + 2\alpha_3 + 2\alpha_4 + \alpha_5\}.\]
In particular, $\dim \Vggp = d_W$. Consequently,  $\dim T_{X_0}\tM^G_{\wm} = d_W$
\end{prop}

\begin{proof}
Note that $G' = \ssE_6$. Consider the $G'$-module $V' := V(\om_1) \oplus V(\om_6)$ and its element $x_0' = v_{\om_1} + v_{\om_6}$. 
Since $V(3\epsilon) \inn \fg\cdot x_0$, we have that
$\Vpggp \isom \Vggp$ as $\Tad$-modules. Observe that $G'_{x_0}=G'_{x_0'}$,

The monoid $p(\<E\>_{\NN}) = \<\om_1, \om_6\>_{\NN}$ is free and $G'$-saturated. By \cite[Theorems 3.1 and 3.10]{bravi&cupit}, $\Vpggp$ is multiplicity-free and its $\Tad$-weights belong to Table 1 in \cite[page 2810]{bravi&cupit}. By Proposition~\ref{prop:propsVggweights} the $\Tad$-weights of $\Vpggp$ also belong to $\Lambda_R \cap \<\om_1, \om_6\>_{\ZZ}$.
A straightforward computation shows that
\begin{equation} \label{eq:4}
\Lambda_R \cap \<\om_1, \om_6\>_{\ZZ} = 
\<\alpha_2 + \alpha_3 + 2\alpha_4 + 2\alpha_5 + 2\alpha_6 , 2\alpha_1
+ \alpha_2 + 2\alpha_3 + 2\alpha_4 + \alpha_5\>_{\ZZ}
\end{equation}
Observe that the support of each of the two generators of $\Lambda_R \cap
\<\om_1, \om_6\>_{\ZZ}$ in equation~(\ref{eq:4}) contains a simple
root not in the support of the other generator. Because the
$\Tad$-weights of $\Vpggp$ belong to $\<\sr\>_{\NN}$, it follows that they belong to $\<\alpha_2 + \alpha_3 + 2\alpha_4 + 2\alpha_5 + 2\alpha_6 , 2\alpha_1
+ \alpha_2 + 2\alpha_3 + 2\alpha_4 + \alpha_5\>_{\NN}$. Because none
of the $\Tad$-weights in \cite[Table 1]{bravi&cupit} supported on a
subdiagram of $\ssE_6$ has a coefficient greater than $2$, it follows
that the $\Tad$-weights of $\Vpggp$ are a subset of $\{\alpha_2 +
\alpha_3 + 2\alpha_4 + 2\alpha_5 + 2\alpha_6, 2\alpha_1
+ \alpha_2 + 2\alpha_3 + 2\alpha_4 + \alpha_5\}$. Since $d_W=2$, this proves the proposition. 
\end{proof}

\subsection{(K19) The module $(\Sp(2n) \times \CC^{\times} \times
  \CC^{\times}, \CC^{2n}\oplus \CC^{2n})$ with $2\leq n$} \label{subsec:case19}
  
For these modules,
\begin{align*}
&E =  \{\lambda_1, \lambda_2, \lambda_3, \lambda_4\};\\
&d_W = 2,
\end{align*}
where
\begin{align*}\lambda_1& :=\om_1 + \epsilon;& \lambda_2&:=\om_1 + \varepsilon';\\ \lambda_3&:=\om_2 + \varepsilon +
\varepsilon';& \lambda_4&:=\varepsilon +\varepsilon'.
\end{align*}  

Note that $G=\overline{G}$ is the only connected group between $\overline{G}'$ and $\overline{G}$ for which this module is spherical, cf.\ Remark~\ref{rem:sphericalforsubgroup}. Therefore, we can assume throughout this section that $G=\overline{G}= \Sp(2n) \times \CC^{\times} \times \CC^{\times}$. 

\begin{prop} \label{prop:case19}
The $\Tad$-module $\tg\ms^G$ is multiplicity-free and has $\Tad$-weight set 
\[\{\alpha_1, \alpha_1 + \gamma\},\]
where $\gamma = 2(\alpha_2+\alpha_3+\ldots+\alpha_{n-1}) + \alpha_n$ if $n>2$ and
$\gamma = \alpha_2$ if $n=2$. 
In particular, $\dim \tg\ms^G = d_W$.
\end{prop}

\begin{proof}
This proof is similar to that of Proposition~\ref{prop:case6}.
Let $\beta$ be a $\Tad$-weight in $\tg\ms^G$. By Proposition~\ref{prop:propsVggweights}, we know that $\beta \in \<\sr\>_{\NN} \cap \<E\>_Z$. A straightforward computation shows that $\Lambda_R \cap \<E\>_{\ZZ} = \<\alpha_1, \gamma\>_{\ZZ}$. Since $\alpha_1$ is not in the support of $\gamma$ and $\alpha_2$ is in the support of $\gamma$ but not in the support of $\alpha_1$, it follows that 
\begin{equation}
\beta \in \<\alpha_1, \gamma\>_{\NN}. \label{eq:beta19}
\end{equation}

By Proposition~\ref{prop:Tadweightnotantidom}, at least one element of $\{\lambda_1, \lambda_2, \lambda_3\}$ must have a positive coefficient in the expression of $\beta$ as a linear combination of the elements of $E$. Since $\lambda_1$ and $\lambda_2$ have codimension $1$, it follows from Proposition~\ref{prop:simprootweight} that if one of them has a positive coefficient, then $\beta = \alpha_1$ and $\beta$ has multiplicity one in $\tg\ms^G$. To finish the proof it is therefore enough to show the following four claims: 
\begin{description}
\item[Claim A] if $\lambda_3$ is the only element of $E$ which has a
positive coefficient in the expression of $\beta$ as a linear combination of the elements of $E$, then $\beta \in \{\alpha_1 + \gamma, \gamma\}$ if $n>2$ and $\beta \in \{\alpha_1 + \gamma, \gamma, 2\alpha_1 + 2\gamma\}$ if $n=2$. 
\item[Claim B] the $\Tad$-weight $\beta = \gamma$ does not occur in $\Vgg$, and therefore also not in the subspace $\tg \ms^G$.
\item[Claim C] suppose $n=2$; the $\Tad$-weight $\beta=2\alpha_1+2\gamma$ does not occur in $\Vgg$, and therefore also not in the supspace $\tg\ms^G$. 
\item[Claim D] the $\Tad$-weight $\beta= \alpha_1 + \gamma$ has multiplicity at most one in $\tg\ms^G$. 
\end{description}
To prove Claim A, we first observe that
\begin{align*}
\alpha_1 &= \lambda_1 + \lambda_2 - \lambda_3;\\
\gamma &= 2\lambda_3 - \lambda_1 - \lambda_2 -\lambda_4.
\end{align*}
It follows from \eqref{eq:beta19} and the hypothesis of Claim A, that there exist $A,B \in \NN$ with $B>0$ and $B\geq A$ such that
\begin{equation}
\beta = A\alpha_1 + B\gamma. \label{eq:beta19bis}
\end{equation}
For $n>2$ the only $\beta$ as in \eqref{eq:beta19bis} that satisfy Proposition~\ref{prop:propsVggweights}(\ref{item:betasum}) are $\gamma$ and $\alpha_1 + \gamma$. For $n=2$, there are three additional such $\beta$, namely $\beta = 2\gamma = 2\alpha_2$, $\beta=\alpha_1+2\gamma=\alpha_1+2\alpha_2$ and $\beta= 2\alpha_1 + 2\alpha_2$. Proposition~\ref{prop:Kostant} tells us that $\beta=2\alpha_2$ and $\beta=\alpha_1+2\alpha_2$ cannot occur as a $\Tad$-weight in $\tg\ms^G$. This finishes the proof of Claim A.

We proceed to Claim B. Let $\beta = \gamma$.  Observe that the $\Tad$-weight $\gamma$ does not occur in the $G$-modules $V(\lambda_1)$, $V(\lambda_2)$ and $V(\lambda_4)$. It occurs in $V(\lambda_3)$ with multiplicity one: the $\Tad$-weight space in $V(\lambda_3)$ of weight $\beta$ is spanned by $X_{-\beta} v_{\lambda_3}$. It follows that $\beta$ occurs with multiplicity one in $V$ and that its weight space is a subspace of $\fg\cdot x_0$. This proves Claim B. 

We move to Claim C. Let $n=2$ and $\beta = 2\alpha_1 + 2\alpha_2$. One deduces from the well-known decompositions into $T$-weight spaces of $V(\om_1)$ and $V(\om_2)$ that the $\Tad$-weight space in $V$ of $\Tad$-weight $\beta$ is a
line $\CC v$ in $V(\lambda_3) \inn V$. We prove Claim C
by contradiction. Indeed, since
$\alpha_2$ is the only simple root such that $\beta - \alpha_2 \in
R^+\cup \{0\}$, we have that  $0\neq [v] \in \Vgg$ would, by Proposition~\ref{prop:propsVggweights}, imply
that $X_{\alpha_2}\cdot v$ is a nonzero element of $\<X_{-(\beta -
  \alpha_2)}\cdot x_0\>_{\CC}$. This is absurd since $X_{-(\beta -
  \alpha_2)}\cdot x_0$ has nonzero projection onto the components $V(\lambda_1)$ and $V(\lambda_2)$ of $V$. Claim C is proved.
 
Finally, we show Claim D. We fix $\beta = \alpha_1 + \gamma$. We will show that the $\Tad$-weight $\beta$ has multiplicity at most one in $\Vgg$. Since $\tg\ms^G \inn \Vgg$, this implies Claim D. First off, we observe that the $\Tad$-weight $\beta$ can only occur in $V(\lambda_1)$, $V(\lambda_2)$ or $V(\lambda_3)$. Let $Z$ be the subspace of $V(\lambda_1) \oplus V(\lambda_2) \oplus V(\lambda_3)$ consisting of $\Tad$-eigenvectors $v$ of $\Tad$-weight $\beta$ that satisfy the following three conditions:

\begin{align}
&X_{\alpha_1} \cdot v \in \<X_{-(\beta - \alpha_1)}x_0\>_{\CC}; \label{eq:20c1} \\
&X_{\alpha_2} \cdot v \in \<X_{-(\beta - \alpha_2)}x_0\>_{\CC}; \text{ and}\label{eq:20c2} \\
&X_{\alpha_k} \cdot v = 0 \text{ for all $k \in \{1,2,\ldots,n\} \setminus \{1,2\}$}. \label{eq:20c3}
\end{align}
By Proposition~\ref{prop:propsVggweights}, every $v \in V^{\beta}$ with $0\neq [v] \in \Vgg$ satisfies \eqref{eq:20c1}, \eqref{eq:20c2} and \eqref{eq:20c3}. 
To show Claim D it is therefore enough to prove that 
\begin{equation}
\dim Z \leq 2, \label{eq:dimZ20leq2}
\end{equation}
since the nonzero vector $X_{-\beta}\cdot x_0$ belongs to $\fg\cdot x_0 \cap Z$. 

To prove the inequality (\ref{eq:dimZ20leq2}) we will make use of the explicit description of $\sp(2n)$ and its root operators given in the proof of \cite[Theorem 2.4.1]{goodman-wallach-gtm}, as well as the notations therein like we did in the proof of Proposition~\ref{prop:case6}; see page~\pageref{sp2ndata}.

Note that 
$$
\beta = \alpha_1 + \gamma = \om_2 = \eps_1 + \eps_2.
$$
We now identify the weight spaces of this $\Tad$-weight in the representations $V(\om_1)$ and $V(\om_2)$ of $\Sp(2n)$. 
A vector in $V(\om_1)$ has $\Tad$-weight $\beta$ if and only if it has $T$-weight $\om_1 - \beta =  - \eps_2$. We identify $V(\om_1)$ with the standard representation $\CC^{2n}$ of $\Sp(2n)$, which has $e_1$ as a highest weight vector. Then the $T$-weight space of weight $- \eps_2$ is the line spanned by
$e_{-2}.$  

A vector in $V(\om_2)$ has $\Tad$-weight $\beta$ if and only if it has $T$-weight $\om_2-\beta=0$. As is well-known, $V(\om_2)$ is the irreducible component of the $\Sp(2n)$-module $\wedge^2 \CC^{2n}$ generated by the highest weight vector $e_1\wedge e_2$. In the larger module $\wedge^2 \CC^{2n}$ the $T$-weight space of weight $0$ is spanned by the following vectors
\begin{equation} \label{eq:20Tweightspom2}
e_1 \wedge e_{-1}, e_2 \wedge e_{-2}, \ldots ,e_n \wedge e_{-n}.
\end{equation} 

It follows from the previous paragraph that if $v\in Z$, then there exist $A, B \in \CC$ and  $C_1, C_2, \ldots , C_n \in \CC$ such that
\begin{equation}
v = A(e_{-2} \otimes v_{\varepsilon}) + B(e_{-2} \otimes v_{\varepsilon'})  
+ \sum_{k=1}^n C_k (e_k \wedge e_{-k} \otimes v_{\varepsilon+\varepsilon'}). 
\end{equation}   
Straightforward computations using the root operators show that conditions (\ref{eq:20c1}), (\ref{eq:20c2}) and (\ref{eq:20c3}) imply that when $n>2$, 
\begin{align}
&A = B = C_3-C_2; \text{ and} \label{eq:20cvinZ1}\\
&C_3 = C_4 = \ldots = C_n,\label{eq:20cvinZ2}
\end{align}
and that when $n=2$,
\begin{align}
&A = B. \label{eq:20cvinZ1nis2}
\end{align}
Either way, this implies that $\dim Z\leq 3$. Note that the vector $v_1 \otimes v_{\varepsilon+\varepsilon'} \in \wedge^2\CC^{2n} \otimes V(\varepsilon+\varepsilon')$, where 
$$v_1 = \sum_{k=1}^n e_k \wedge e_{-k},$$ satisfies the equations (\ref{eq:20cvinZ1}) and (\ref{eq:20cvinZ2}). It is straightforward to check that $v_1$ is a highest weight vector. It follows that  $v_1$ is an element of the $\Sp(2n)$-stable complement to $V(\om_2)$ in $\wedge^2 \CC^{2n}$. Consequently, the line spanned by $v_1 \otimes v_{\varepsilon+\varepsilon'}$ is not contained in $Z$, and $\dim Z \leq 2$. This proves Claim D, and the proposition.
\end{proof}

\begin{remark}
Proceeding as in the proof of \cite[Proposition
5.14]{degensphermod}, one can show that in fact $\Vggp$ is a multiplicity-free $\Tad$-module with the same $\Tad$-weight set as $T_{X_0}\ms$.  
\end{remark}

\subsection{(K20) The modules $((\Sp(2n)\times \CC^{\times}) \times \GL(2), (\CC^{2n}\otimes \CC^2)\oplus \CC^2)$ with $2 \leq n$.}  

For these modules,
\begin{align*}
&E =  \{\lambda_1, \lambda_2, \lambda_3, \lambda_4, \lambda_5\};\\
&d_W = 3,
\end{align*}
where
\begin{align*}\lambda_1& :=\om_1';& \lambda_2&:=\om_1 + \varepsilon + \om_1';& \lambda_3&:=\om_1+\varepsilon +\om_2';\\ \lambda_4&:=\om_2 + 2\varepsilon + \om_2';& \lambda_5&:=2\varepsilon + \om_2'.
\end{align*}
We remark that $G = \overline{G}$ is the only connected group between $\overline{G}'$ and $\overline{G}$ for which these modules are spherical, cf.\ Remark~\ref{rem:sphericalforsubgroup}. Therefore, we assume $G = \overline{G} = \Sp(2n) \times \CC^{\times} \times \GL(2)$ throughout this section.

In this section we will prove the following proposition.
\begin{prop} \label{prop:case20}
The $\Tad$-module  $T_{X_0}\tM^G_{\wm}$ is multiplicity-free. Its $\Tad$-weight set is
\begin{equation}
\{\alpha_1, \alpha_1+\gamma, \alpha_1'\}, \label{eq:case20weights}
\end{equation}
where $\gamma_1= \alpha_2$ if $n=2$ and $\gamma = 2(\alpha_2+\alpha_3+\ldots+\alpha_{n-1}) + \alpha_n$ if $n>2$. 
In particular, 
$\dim T_{X_0}\tM^G_{\wm} = d_W$.
\end{prop}
\begin{proof}
The argument is very similar to that of Proposition~\ref{prop:case6}. 
Let $\beta$ be a $\Tad$-weight in $\tg\ms^G$. 
It follows from Proposition~\ref{prop:Tadweightnotantidom} that at
least one $\lambda \in E$ has a positive coefficient in the
expression of $\beta$ as a $\ZZ$-linear combination of elements of
$E$. Note that all elements of $E$ except $\lambda_4$ have codimension
$1$. In particular, if $\lambda \neq \lambda_4$, then it follows
from Proposition~\ref{prop:simprootweight} that $\beta$ is a simple
root belonging to the set~\eqref{eq:case20weights}, and that its weight space has
dimension one. Consequently, to prove the proposition, what remains is to show
the following two claims: 
\begin{description}
\item[Claim A] if $\lambda_4$ is the only element of $E$ which has a
positive coefficient in the expression of $\beta$, then $\beta= \alpha_1 + \gamma$. 
\item[Claim B] the $\Tad$-weight $\beta= \alpha_1 + \gamma$ has multiplicity at most one in $\tg\ms^G$. 
\end{description}

We begin with Claim A. Recall from Proposition~\ref{prop:propsVggweights} that $\beta \in \<E\>_{\ZZ} \cap \<\Pi\>_{\NN}$. Straightforward computations show that
\begin{equation} \label{eq:case20lat}
\<E\>_{\ZZ} \cap \Lambda_R =
\<\alpha_1, \alpha_1',\gamma\>_{\ZZ}
\end{equation}
and that
\begin{align*}
\alpha_1 &= \lambda_2 + \lambda_3 - \lambda_1 - \lambda_4; \\
\alpha'_1 &= \lambda_1 + \lambda_2 - \lambda_3;\\
\gamma &= 2\lambda_4 + \lambda_1 - \lambda_2 - \lambda_3 - \lambda_5.
\end{align*}
Let $K$ be the basis of $\<E\>_{\ZZ} \cap \Lambda_R$ given in equation~\eqref{eq:case20lat}. 
Since $\beta \in \<\sr\>_{\NN}$ and all elements of $K$ contain a simple root in their support, which is not in the support of any other element of $K$, it follows that $\beta \in \<K\>_{\NN}$. Therefore, there exist $A,B, C \in \NN$ such that 
\begin{equation}
\beta = A \alpha_1 + B \alpha'_1 + C \gamma. 
\end{equation}
From the hypothesis of Claim A, it follows that
\begin{equation}
\begin{cases}
-A+B+C\leq 0;\\
A+B-C\leq 0;\\ 
A-B-C\leq 0.
 \label{eq:c20ineqs}
\end{cases}
\end{equation}
Adding the first two inequalities in~\eqref{eq:c20ineqs} yields that $B=0$. After substituting $B=0$, the first two inequalities yield that $A=C$. It follows that $\beta \in \<\alpha_1 + \gamma\>_{\NN}$. Using that $\beta$ is the sum of a simple root and an element of $\pr \cup \{0\}$ (see Proposition~\ref{prop:propsVggweights}) it follows that $\beta = \alpha_1 + \gamma$ when $n>2$ and that $\beta \in \{\alpha_1 + \gamma, 2\alpha_1 + 2\gamma\}$ when $n=2$. With an argument like that for Claim C in the proof of Proposition~\ref{prop:case19}, one shows that $2\alpha_1 + 2\gamma$ cannot occur as a $\Tad$-weight in $\Vgg$ when $n=2$.  This proves Claim A.

The argument for Claim B is the same as that for Claim D in Proposition~\ref{prop:case19} above.
\end{proof}

\subsection{(K22) The modules $((\Sp(2m) \times \CC^{\times}) \times \SL(2) \times \GL(n),
  (\CC^{2m} \otimes \CC^2) \oplus (\CC^2 \otimes \CC^n))$ with $2 \leq
  m,n$}
  
For these modules,
\begin{align*}
&E =  \{\lambda_1, \lambda_2, \lambda_3, \lambda_4, \lambda_5, \lambda_6\};\\
&d_W = 4,
\end{align*}
where
\begin{align*}\lambda_1& :=\om_1 + \varepsilon + \om';& \lambda_2&:=\om' + \om_1'';& \lambda_3&:=\om_1+\varepsilon +\om_1'';\\ \lambda_4&:=\om_2 + 2\varepsilon;& \lambda_5&:= \om_2'' ; & \lambda_6&:=2\varepsilon.
\end{align*}

In this section we will prove the following proposition.
\begin{prop} \label{prop:case22}
The $\Tad$-module  $T_{X_0}\tM^G_{\wm}$ is multiplicity-free. Its $\Tad$-weight set is
\begin{equation}
\{\alpha_1, \alpha_1+\gamma,\alpha', \alpha_1''\}, \label{eq:case22weights}
\end{equation}
where $\gamma= \alpha_2$ if $m=2$ and $\gamma = 2(\alpha_2+\alpha_3+\ldots+\alpha_{m-1}) + \alpha_m$ if $m>2$. 
In particular, 
$\dim T_{X_0}\tM^G_{\wm} = d_W$.
\end{prop}

\begin{proof}
The argument is very similar to that of Propositions \ref{prop:case6}, \ref{prop:case19} and \ref{prop:case20}. 
Let $\beta$ be a $\Tad$-weight in $\tg\ms^G$. 
It follows from Proposition~\ref{prop:Tadweightnotantidom} that at
least one $\lambda \in E$ has a positive coefficient in the
expression of $\beta$ as a $\ZZ$-linear combination of elements of
$E$. Note that all elements of $E$ except $\lambda_4$ and $\lambda_5$ have codimension
$1$. In particular, if $\lambda \not\in \{\lambda_4,\lambda_5\}$, then it follows
from Proposition~\ref{prop:simprootweight} that $\beta$ is a simple
root belonging to the set~\eqref{eq:case22weights}, and that its weight space has
dimension one. 

Consequently, to prove the proposition, what remains is to show
the following two claims: 
\begin{description}
\item[Claim A] if $\lambda_4$ or $\lambda_5$ are the only elements of $E$ which have a
positive coefficient in the expression of $\beta$, then $\beta= \alpha_1 + \gamma$. 
\item[Claim B] the $\Tad$-weight $\beta= \alpha_1 + \gamma$ has multiplicity at most one in $\tg\ms^G$. 
\end{description}

We begin with Claim A. Recall from Proposition~\ref{prop:propsVggweights} that $\beta \in \<E\>_{\ZZ} \cap \<\Pi\>_{\NN}$. Consequently $\beta \in p(\<E\>_{\ZZ}) \cap \Lambda_R$. Straightforward computations show that
\begin{equation} \label{eq:case22lat}
p(\<E\>_{\ZZ}) \cap \Lambda_R =
\<\alpha_1,\alpha' \alpha_1'',\gamma\>_{\ZZ}
\end{equation}
and that
\begin{align*}
\alpha_1 &= \lambda_1 + \lambda_3 - \lambda_2 - \lambda_4; \\
\alpha' &= \lambda_1 + \lambda_2 - \lambda_3;\\
\alpha_1''&= \lambda_2 + \lambda_3 - \lambda_1 - \lambda_5;\\
\gamma &= 2\lambda_4 + \lambda_2 - \lambda_1 - \lambda_3 - \lambda_6.
\end{align*}
Let $K$ be the basis of $p(\<E\>_{\ZZ}) \cap \Lambda_R$ given in equation~\eqref{eq:case22lat}. 
Since $\beta \in \<\sr\>_{\NN}$ and all elements of $K$ contain a simple root in their support, which is not in the support of any other element of $K$, it follows that $\beta \in \<K\>_{\NN}$. Therefore, there exist $A_1, A_2, A_3, A_4 \in \NN$ such that 
\begin{equation}
\beta = A_1 \alpha_1 + A_2 \gamma + A_3 \alpha' + A_4 \alpha_1''. 
\end{equation}
From the hypothesis of Claim A, it follows that
\begin{equation}
\begin{cases}
A_1-A_2 + A_3 - A_4\leq 0\\
-A_1+A_2+A_3 +A_4\leq 0\\ 
A_1-A_2-A_3+A_4\leq 0
 \label{eq:c22ineqs}
\end{cases}
\end{equation}
Adding the last two inequalities in~\eqref{eq:c22ineqs} yields that $A_4=0$. After substituting $A_4=0$, the first two inequalities yield that $A_3=0$ and then that $A_1=A_2$. It follows that $\beta \in \<\alpha_1 + \gamma\>_{\NN}$. Using that $\beta$ is the sum of a simple root and an element of $\pr \cup \{0\}$ (see Proposition~\ref{prop:propsVggweights}) it follows that $\beta = \alpha_1 + \gamma$  when $m>2$ and that $\beta \in \{\alpha_1 + \gamma, 2\alpha_1 + 2\gamma\}$ when $m=2$. With an argument like that for Claim C in the proof of Proposition~\ref{prop:case19}, one shows that $2\alpha_1 + 2\gamma$ cannot occur as a $\Tad$-weight in $\Vgg$ when $m=2$. This proves Claim A.

We now proceed to Claim B. We fix $\beta = \alpha_1 + \gamma$. We will show that the $\Tad$-weight $\beta$ has multiplicity at most one in $\Vgg$. Since $\tg\ms^G \inn \Vgg$, this implies Claim B. First off, we note that the $\Tad$-weight $\beta$ only occurs in $V(\lambda_1)$, $V(\lambda_3)$ and in $V(\lambda_4)$, since $\beta$ belongs to the root lattice of $\Sp(2n)$. Let $Z$ be the subspace of $V(\lambda_1) \oplus V(\lambda_3) \oplus V(\lambda_4)$ consisting of $\Tad$-eigenvectors $v$ of $\Tad$-weight $\beta$ that satisfy the following three conditions:

\begin{align}
&X_{\alpha_1} \cdot v \in \<X_{-(\beta - \alpha_1)}x_0\>_{\CC}; \label{eq:22c1} \\
&X_{\alpha_2} \cdot v \in \<X_{-(\beta - \alpha_2)}x_0\>_{\CC}; \text{ and}\label{eq:22c2} \\
&X_{\alpha_k} \cdot v = 0 \text{ for all $k \in \{1,2,\ldots,n\} \setminus \{1,2\}$}. \label{eq:22c3}
\end{align}
By Proposition~\ref{prop:propsVggweights}, every $v \in V^{\beta}$ with $0\neq [v] \in \Vgg$ satisfies \eqref{eq:22c1}, \eqref{eq:22c2} and \eqref{eq:22c3}.  
To show Claim B it is therefore enough to prove that 
\begin{equation}
\dim Z \leq 2, \label{eq:dimZ22leq2}
\end{equation}
since the nonzero vector $X_{-\beta}\cdot x_0$ belongs to $\fg\cdot x_0 \cap Z$. The proof of (\ref{eq:dimZ22leq2}) is the same as that of (\ref{eq:dimZ20leq2}). 
\end{proof}

\subsection{(K23) The modules $((\Sp(2m) \times \CC^{\times}) \times \SL(2) \times (\Sp(2n)
  \times \CC^{\times}),
  (\CC^{2m} \otimes \CC^2) \oplus (\CC^2 \otimes \CC^{2n}))$ with $2 \leq
  m,n$}

For these modules,
\begin{align*}
&E =  \{\lambda_1, \lambda_2, \lambda_3, \lambda_4, \lambda_5, \lambda_6, \lambda_7\};\\
&d_W = 5,
\end{align*}
where
\begin{align*}\lambda_1& :=\om_1 + \varepsilon + \om';& \lambda_2&:=\om' + \om_1'' + \varepsilon';& \lambda_3&:=\om_1+\varepsilon +\om_1'' + \varepsilon';& \lambda_4&:=\om_2 + 2\varepsilon;\\ \lambda_5&:= \om_2''+2\varepsilon' ; & \lambda_6&:=2\varepsilon; & \lambda_7 &:= 2\varepsilon'.
\end{align*}

We remark that $G = \overline{G}$ is the only connected group between $\overline{G}'$ and $\overline{G}$ for which these modules are spherical, cf.\ Remark~\ref{rem:sphericalforsubgroup}.  Therefore, we assume $G = \overline{G} = (\Sp(2m) \times \CC^{\times}) \times \SL(2) \times (\Sp(2n)
  \times \CC^{\times})$ throughout this section.

In this section we will prove the following proposition.
\begin{prop} \label{prop:case23}
The $\Tad$-module  $T_{X_0}\tM^G_{\wm}$ is multiplicity-free. Its $\Tad$-weight set is
\begin{equation}
\{\alpha_1, \alpha_1+\gamma,\alpha', \alpha_1'', \alpha_1''+\gamma''\}, \label{eq:case23weights}
\end{equation}
where 
\begin{align*}
\gamma&= \begin{cases} \alpha_2 &\text{if $m=2$};
\\ 2(\alpha_2+\alpha_3+\ldots+\alpha_{m-1}) + \alpha_m &\text{if $m>2$};
\end{cases}\\
\gamma''&= \begin{cases} \alpha''_2 &\text{if $n=2$};
\\ 2(\alpha''_2+\alpha''_3+\ldots+\alpha''_{n-1}) + \alpha''_n &\text{if $n>2$}.
\end{cases}
\end{align*}
In particular, 
$\dim T_{X_0}\tM^G_{\wm} = d_W$.
\end{prop}

\begin{proof}
The argument is very similar to that of Propositions \ref{prop:case6}, \ref{prop:case20} and \ref{prop:case22}. 
Let $\beta$ be a $\Tad$-weight in $\tg\ms^G$. 
It follows from Proposition~\ref{prop:Tadweightnotantidom} that at
least one $\lambda \in E$ has a positive coefficient in the
expression of $\beta$ as a $\ZZ$-linear combination of elements of
$E$. Note that all elements of $E$ except $\lambda_4$ and $\lambda_5$ have codimension
$1$. In particular, if $\lambda \not\in \{\lambda_4,\lambda_5\}$, then it follows
from Proposition~\ref{prop:simprootweight} that $\beta$ is a simple
root belonging to the set~\eqref{eq:case22weights}, and that its weight space has
dimension one. 

Consequently, to prove the proposition, what remains is to show
the following two claims: 
\begin{description}
\item[Claim A] if $\lambda_4$ or $\lambda_5$ are the only element of $E$ which have a
positive coefficient in the expression of $\beta$, then $\beta= \alpha_1 + \gamma$ or $\beta= \alpha_1''+\gamma''$. 
\item[Claim B] the $\Tad$-weights $\beta=\alpha_1 + \gamma$ and $\beta''=\alpha_1''+\gamma_1''$ have multiplicity at most one in $\tg\ms^G$. 
\end{description}

We begin with Claim A. Recall from Proposition~\ref{prop:propsVggweights} that $\beta \in \<E\>_{\ZZ} \cap \<\Pi\>_{\NN}$. Straightforward computations show that
\begin{equation} \label{eq:case23lat}
\<E\>_{\ZZ} \cap \Lambda_R =
\<\alpha_1,\gamma,\alpha', \alpha_1'',\gamma''\>_{\ZZ}
\end{equation}
and that
\begin{align*}
\alpha_1 &= \lambda_1 + \lambda_3 - \lambda_2 - \lambda_4; \\
\alpha' &= \lambda_1 + \lambda_2 - \lambda_3;\\
\alpha_1''&= \lambda_2 + \lambda_3 - \lambda_1 - \lambda_5;\\
\gamma &= 2\lambda_4 + \lambda_2 - \lambda_1 - \lambda_3 - \lambda_6;\\
\gamma''&= 2\lambda_5+\lambda_1 - \lambda_2 - \lambda_3-\lambda_7.
\end{align*}
Let $K$ be the basis of $\<E\>_{\ZZ} \cap \Lambda_R$ given in equation~\eqref{eq:case23lat}. 
Since $\beta \in \<\sr\>_{\NN}$ and all elements of $K$ contain a simple root in their support, which is not in the support of any other element of $K$, it follows that $\beta \in \<K\>_{\NN}$. Therefore, there exist $A_1, A_2, A_3, A_4, A_5 \in \NN$ such that 
\begin{equation}
\beta = A_1 \alpha_1 + A_2 \gamma + A_3 \alpha' + A_4 \alpha_1''+A_5 \gamma''. 
\end{equation}
From the hypothesis of Claim A, it follows that
\begin{equation}
\begin{cases}
A_1-A_2 + A_3 - A_4+A_5\leq 0;\\
-A_1+A_2+A_3 +A_4-A_5\leq 0;\\ 
A_1-A_2-A_3+A_4-A_5\leq 0.
 \label{eq:c23ineqs}
\end{cases}
\end{equation}
Adding the first two inequalities in~\eqref{eq:c23ineqs} yields that $A_3=0$. Adding the first and the third inequality tells us that $A_1 \leq A_2$, while adding the second and third gives $A_4\leq A_5$. Moreover, after substituting $A_3=0$, the first two inequalities also give us that $A_2-A_1 = A_5-A_4$. Put $C := A_2 - A_1 = A_5 - A_4$. 
Then $C \in \NN$ and 
\begin{align}
\beta&= A_1\alpha_1 + (A_1 + C)\gamma + A_4 \alpha_1'' + (A_4+C)\gamma''\\
&= (A_1+2C)\lambda_4 +(A_4+2C)\lambda_5 - 2C\lambda_3 - (A_1+C)\lambda_6 - (A_4+C)\lambda_7. \label{eq:beta23}
\end{align}
By Proposition~\ref{prop:Kostant}, it follows from (\ref{eq:beta23}) that $A_1+A_4+4C \leq 2$. This implies that $C=0$. The inequality $A_1+A_4 \leq 2$ has five solutions in $\NN \times \NN$. This implies that $\beta \in \{\beta_1,\beta_2,\ldots,\beta_5\}$ where $\beta_1=\alpha_1+\gamma$, $\beta_2=\alpha''_1+\gamma''$, $\beta_1=\alpha_1+\gamma+\alpha''_1+\gamma''$, $\beta_4=2\alpha_1+2\gamma$ and $\beta_5=2\alpha''_1+2\gamma''$.
We cannot have $\beta = \beta_3$ because $\beta_3$ does not belong to $\sr + (\pr \cup \{0\})$. If $m>2$, then $\beta \neq \beta_4$ for the same reason. If $m=2$, then an argument like that for Claim C in the proof of Proposition~\ref{prop:case19} shows that $\beta\neq \beta_4$. If $n>2$, then $\beta \neq \beta_5$ because $\beta_5 \notin \sr + (\pr \cup \{0\})$. If $n=2$, then $\beta \neq \beta_5$ by an argument like that for Claim C in the proof of Proposition~\ref{prop:case19}. This proves Claim A.  

The argument for Claim B is the same as that for Claim B in the proof of Proposition~\ref{prop:case20}, except that one has to go through it twice: first for $\beta = \alpha_1 + \gamma$ and then for $\beta'' = \alpha_1'' + \gamma''$. This finishes the proof.
\end{proof}

\subsection{(K24) The module $(\Spin(8) \times \CC^{\times} \times
  \CC^{\times}, \CC^{8}_+\oplus \CC^8_-)$} \label{subsec:case24}
Here 
\begin{align*}
&E =  \{\om_3 + \epsilon, \om_4 + \epsilon', \om_1 + \epsilon +
\epsilon', 2\epsilon, 2\epsilon'\};\\
&d_W = 3.
\end{align*}

\begin{prop}
The $\Tad$-module $\Vggp$ is multiplicity-free and has $\Tad$-weight set 
\[\{\alpha_1 + \alpha_2 + \alpha_3, \alpha_1+\alpha_2 + \alpha_4, \alpha_2 + \alpha_3 + \alpha_4\}.\]
In particular, $\dim \Vggp = d_W$. Consequently,  $\dim T_{X_0}\tM^G_{\wm} = d_W$
\end{prop}

\begin{proof}
Note that $G' = \Spin(8)$. Consider the $G'$-module $V' := V(\om_1)
\oplus V(\om_3) \oplus V(\om_4)$ and its element $x_0' = v_{\om_1} +
v_{\om_3} + v_{\om_4}$. 
Since $V(2\epsilon)$ and $V(2\epsilon')$ are subspaces of $\fg\cdot x_0$, we have that
$\Vpggp \isom \Vggp$ as $\Tad$-modules. Observe that $G'_{x_0} = G'_{x'_0}$. 

The monoid $p(\<E\>_{\NN}) = \<\om_1, \om_3\, \om_4\>_{\NN}$ is free and $G'$-saturated. By \cite[Theorems 3.1 and 3.10]{bravi&cupit}, $\Vpggp$ is multiplicity-free and its $\Tad$-weights belong to Table 1 in \cite[page 2810]{bravi&cupit}. By Proposition~\ref{prop:propsVggweights}, the $\Tad$-weights of $\Vpggp$ also belong to $\Lambda_R \cap \<\om_1, \om_3,\om_4\>_{\ZZ}$.
A straightforward computation shows that
\begin{equation} \label{eq:5}
\Lambda_R \cap \<\om_1, \om_3, \om_4\>_{\ZZ} = 
\{a\alpha_1 + b \alpha_2 + c \alpha_3 + d\alpha_4 \mid a,b,c,d \in \ZZ
\text{ and }2b = a+c+d \}.
\end{equation}

Let $\gamma$ be a $\Tad$-weight of $\Vpggp$. It follows from
equation~(\ref{eq:5}) that
\begin{equation} \label{eq:6}
|\mathrm{supp}(\gamma)| \ge 2 \text{ and } \alpha_2 \in \mathrm{supp}(\gamma).
\end{equation}
There are twelve $\Tad$-weights in \cite[Table 1, page
2810]{bravi&cupit} that satisfy (\ref{eq:6}). Six of them are
\[
\alpha_1+ \alpha_2, \alpha_2 + \alpha_3, \alpha_2 + \alpha_4,
\alpha_1+ 2\alpha_2 + \alpha_3, \alpha_1 + 2\alpha_2 + \alpha_4,
2\alpha_2 + \alpha_3 +\alpha_4,
\]
but $\gamma$ cannot be among these since they do not belong to
$\Lambda_R \cap \<\om_1, \om_3, \om_4\>_{\ZZ}$ by equation~(\ref{eq:5}). 

Three more $\Tad$-weights in \cite[Table 1]{bravi&cupit} that satisfy (\ref{eq:6}) are
\begin{align*}
\gamma_1 &:= 2\alpha_1 + 2\alpha_2 + \alpha_3 + \alpha_4 = 2\om_1\\
\gamma_3 &:= \alpha_1 + 2\alpha_2 + 2\alpha_3 + \alpha_4 = 2\om_3 \\
\gamma_4&:= \alpha_1 + 2\alpha_2 + \alpha_3 + 2\alpha_4 = 2\om_4
\end{align*}
We claim that none of them is a  $\Tad$-weight in $\Vpggp$. Let $i \in
\{1,3,4\}$.  We will argue by
contradiction that $\gamma_i$ is not a $\Tad$-weight in $\Vpggp$.
Assume $v \in V'$ is a $\Tad$-eigenvector of weight $\gamma_i$ such
that $[v]$ is nonzero in $\Vpggp$. 
Note that $\alpha_i$ is the only simple root $\beta$
such that $\gamma_i-\beta$ is in $R^+ \cup \{0\}$. 
By Proposition~\ref{prop:propsVggweights} this implies that
\begin{equation}\label{eq:7}
X_{\alpha_i} \cdot v \in \<X_{-(\gamma_i - \alpha_i)} x'_0\>_{\CC} \setminus\{0\}.
\end{equation} 
Because $\<(\gamma_i - \alpha_i)^{\vee}, \cdot\>$ is nonzero on $\om_1,
\om_3$ and $\om_4$, we have that
$X_{-(\gamma_i - \alpha_i)} x'_0$ has nonzero projection on the three
summands $V(\om_1)$, $V(\om_3)$ and $V(\om_4)$ of $V'$. On the other
hand, 
one checks with \emph{LiE} that $\gamma_i$ does not
occur as a $\Tad$-weight in all three components of $V'$ (see the
proof of Proposition~\ref{prop:case12} for the code of a similar compuation in \emph{LiE}).
This contradiction with equation~(\ref{eq:7}) proves the
claim. 

The remaining three $\Tad$-weights in \cite[Table 1]{bravi&cupit} that
satisfy equation~(\ref{eq:6}) are the three weights listed in the
proposition. Since $d_W = 3$ this proves the proposition. 
\end{proof}

\section* {APPENDIX:  Computing the invariants of the modules in the family K5}
During the work for the present paper, we also developed a different technique
which explicitly computes the invariants in $\Vgg$. The main idea is to use theoretical and elementary 
arguments to  reduce the problem to the
study of the smallest case for the parameter $n$  (or the parameters $(n,m)$),
and then do a direct computation for the smallest case.
In this appendix we present the method for the K5 family.  The method 
works equally well for the study of the remaining infinite families.

\subsection {Notation and Generalities about $\Sp_{2n}$}  \label{sec!notationsforsp2n}
To accommodate the computational and explicit nature of this appendix, the notation used here is different from that used in the rest of the paper. 
Assume $n \geq 1$ is a positive integer. Consider the vector space
$\CC^{2n}$ with basis $e_1, \dots , e_{2n}$. We also set
$f_i = e_{2n+1-i}$ for $1 \leq i \leq n$.

We define a nondegenerate
skewsymmetric bilinear form $\Omega : \CC^{2n} \times  \CC^{2n} \to  \CC$
by
\[
   \Omega (e_i, e_j) = \Omega (f_i, f_j)  = 0,  \quad   \quad  \Omega (e_i, f_j) = \delta_{ij}, 
    \quad     \quad   \Omega (f_i, e_j) = - \delta_{ij}
\]
for $1 \leq i,j \leq n$, where $\delta_{ij}$ denotes the Kronecker delta function. 
By definition $\Sp_{2n}$ consists of the linear automorphisms $g$ of $\CC^{2n}$ which
have the property $\Omega (g(v), g(w)) = \Omega (v, w)$ for all $v,w \in \CC^{2n}$. 

We denote by  $\sp_{2n}$ the Lie algebra of $\Sp_{2n}$. According to 
 \cite[p.72, Eq (2.8)]{goodman-wallach-gtm}
it has a basis
\[
   \{ a_{ij}, \; 	b_{kl}, \;  c_{kl} : 1 \leq i, j \leq n, \; 1 \leq k \leq l  \leq n \}
\]
defined as follows: $ a_{ij} = ( e_j \mapsto e_i, f_i \mapsto -f_j)$ where the notation
means that $ a_{ij} (e_j) = e_i, a_{ij} (f_i) = -f_j$, 
$a_{ij} (e_t) = 0$  if $ 1 \leq t \leq n$ and $t \not= j$,  and
$a_{ij} (f_t) = 0$ if $ 1 \leq t \leq n$ and $t \not= i$. With the same notational
convention $ b_{kl} = ( e_l \mapsto f_k, e_k \mapsto f_l)$  and
$ c_{kl} = ( f_l \mapsto e_k, f_k \mapsto e_l)$.

\subsection {Notation for the K5  example}

 By definition, for  $n \geq 2$,   K5 with parameter $n$, or more simply
K5($n$),  is the  K5 family in  List~(\ref{KnopL})  with group $\Sp_{2n} \times
 \GL_2$ and $W = \CC^{2n} \otimes \CC^2$.
    Fix  $n \geq 2$.    Set   $G  =   \Sp_{2n} \times  \GL_2$.

    We denote by    $\epsilon_1, \dots  , \epsilon_n$    the standard basis of the weight lattice of  $\Sp_{2n}$ and by
                  $\epsilon_1',  \epsilon_2'$    the standard
                             basis of the weight lattice of  $\GL_2$.
      For $1 \leq i \leq n$  we denote by   $\omega_i$  the $i$-th fundamental weight of $\Sp_{2n}$,
      and for $1 \leq i \leq 2$ we denote   by   $\omega_i'$  the $i$-th fundamental weight of $\GL_2$.
     We set:   
\[
          V_1' =  V(\omega_1) \otimes  V(\omega_1') ,  \quad    
           V_2'  =  \wedge^2 \CC^{2n}  \otimes  V(\omega_2'), \quad 
          V_3'  =    V(\omega_2'),  \quad
          V' = \oplus _{i=1}^3  V_i'. 
\]
We also set  
      \[  V_2  =   V(\omega_2)  \otimes  V(\omega_2') \subset  V_2',  \quad V_i = V_i' \;   \text{ for } i \not= 2,  \;
          \text { and} \;  \;  V  = \oplus _{i=1}^3  V_i.
     \]

  We define, for $1 \leq i \leq 3$, dominant weights $\lambda_i$ of $G$
  by $V_i = V(\lambda_i)$. Hence $\lambda_1 = \omega_1 + \omega_1'$,
  $\lambda_2 = \omega_2 + \omega_2'$, $\lambda_3 =  \omega_2'$.
We denote by $T$ the diagonal maximal torus of $G$.
We define an action    $\rhoprime :  T \times V' \to V'$ by   
    $ \rhoprime (t,\sum_{i=1}^3 w_i) = \sum_{i=1}^3 \lambda_i(t)  t^{-1} \cdot w_i$,  
for $t \in T$ and $w_i \in V_i'$ for $1 \leq i \leq 3$.   It is clear that for $t \in T$ and $w \in V$ we have
$ \rhoprime (t,w)  = \alpha (t,w)$, where $\alpha: T \times V \to V$ is 
the action  defined in   \cite  [Definition 2.11]{degensphermod}.

   For the $\CC^{2n}$ that $\Sp_{2n}$ acts we fix a basis $e_1, \dots , e_{2n}$ and define $\Omega, f_i, 
   a_{ij}, b_{kl}, c_{kl}$ as in Subsection~\ref{sec!notationsforsp2n}. 
   For the $\CC^2$ that $\GL_2$ acts we fix a basis $g_1,g_2$ and 
   define a basis  $\{ d_{pq} : 1 \leq p, q \leq 2 \}$ of $\gl_2$ by 
   $d_{pq}(e_a) =  e_p$ if $q = a$ and $0$ otherwise.  
   Then the set 
  \[
   \{ a_{ij}, \;  b_{kl}, \;  c_{kl},  \; d_{pq} : 1 \leq i,j \leq n, \; 1 \leq k \leq l \leq n,
       \;  1 \leq p, q \leq 2 \}
\]
is a basis of $\fg$ which we call the standard basis.

We set   $\Aset = \{ 1,2,  2n-1, 2n \}$ and
$ \;  H = \sum_{i=1}^n e_i \wedge f_i  \in  \wedge^2 \CC^{2n}, \quad 
     H_{s} = \sum_{i=1}^2 e_i \wedge f_i  \in  \wedge^2 \CC^{2n}, \quad
      x_0  =  e_1  \otimes g_1 +    e_1 \wedge e_2 \otimes g_1 \wedge g_2    + 
              g_1 \wedge g_2. \;$
For $1 \leq i \leq 3$ we denote by $v_i$ the component of $x_0$ that is in $V_i$.
For example, $v_3 = g_1 \wedge g_2$.

We say that $v \in V'$ is an invariant if $[v] \in \Vggpr$.
For  $1 \leq i \leq 2n$ we define that the index of $e_i$ is $i$. 
Since $f_i = e_{2n+1-i}$,  we also define that the index of $f_i$ is $2n+1-i$.  
We denote by $X_0 \subset V$
the Zariski closure of the $G$-orbit of $x_0$.

We fix the following basis of $V'$   which we call the monomial basis: \\
   $\phantom{===} e_i  \otimes g_p  :  1 \leq i \leq 2n ,   1 \leq p \leq 2,   \quad 
     e_i \wedge e_j  \otimes g_1 \wedge g_2 :  1 \leq i < j \leq 2n , \quad   g_1 \wedge g_2$. \\
For $v \in V'$, the terms of $v$ are by definition the 
nonzero monomial terms of the (unique) expression of $v$ 
as a linear combination of the monomial basis. For example, the terms of
$ 2 e_2 \otimes g_1 + 7 g_1 \wedge g_2$ are 
$2e_2 \otimes g_1$ and $7 g_1 \wedge g_2$.

We denote  by  $V_a'$,  the linear span of the subset of the monomial 
basis of $V'$ where all indices appearing for $e_i$
are  in $\Aset$. In more detail,   $V_a'$ is the linear span of       \\
   $\phantom{=} e_i  \otimes g_p  :  i \in \Aset, \;  1 \leq p \leq 2 , \quad
    e_i \wedge e_j  \otimes g_1 \wedge g_2  :  1 \leq i < j \leq 2n ,\;  \{ i, j \} \subset \Aset,  \quad g_1 \wedge g_2$. \\
We denote by   $V_b'$   the linear span of the remaining elements of the 
monomial basis of $V'$.

We define a Lie subalgebra $Z_a  \subset \fg$ and a vector subspace $Z_b \subset \fg$ such that, as vector
space, $\fg$ is the direct sum of $Z_a$ and $Z_b$. Namely, we set $Z_a$ to be the linear span of
\[
   \{ a_{ij}, \; b_{kl} \;,  \; c_{kl},  \; d_{ij} : \quad 1 \leq i,j \leq 2, \; 1 \leq k \leq l \leq 2 \}
\]  
and  $Z_b$ to be the linear span of the remaining elements of the standard 
basis of $\fg$. Clearly $Z_a$ is in a natural way isomorphic to $\sp_4 \oplus \gl_2$
which is the Lie algebra  of the group  $\Sp_4 \times \GL_2$ of the
case K5($2$).

An easy direct computation proves the following proposition. 

\begin {prop}    \label {prop!aboutZaVasubsetVaForK5}
  We have   $Z_a \cdot V_a' \subset V_a' , \; Z_a \cdot V_b' \subset V_b'$ and
                  $Z_b \cdot V_a' \subset V_b'$. 
                  As a corollary,  $Z_a \cdot x_0  \subset  V_a'$ and $Z_b \cdot x_0  \subset  V_b'$.
\end {prop}

\begin {cor}    \label {cor!firstcorollary23ForK5}
   Assume   $v \in \fg \cdot x_0$.   Write    $   v = v_a + v_b$ 
    with   $v_a \in V_a',  v_b \in V_b'$.   Then there exist 
    $z_a \in Z_a$   and   $z_b \in Z_b$   with 
                   $v_a = z_a \cdot x_0$   and    $v_b = z_b \cdot x_0$.
\end {cor}

\begin {proof}
     There exists $z \in \fg$ with  $v = z \cdot x_0$.   Write $z = z_a + z_b$   with
      $z_a \in Z_a$   and   $z_b \in Z_b$, then $v = z_a \cdot x_0 +  z_b \cdot x_0$.
   Using Proposition~\ref  {prop!aboutZaVasubsetVaForK5}  $z_a \cdot x_0 \in V_a'$  and   $z_b \cdot x_0 \in V_b'$.  
   Since  $V_a' \cap V_b'  = \{ 0 \}$ we get
               $v_a = z_a \cdot x_0$   and    $v_b = z_b \cdot x_0$.  
\end{proof}

We denote by  $ \fg_{x_0}$ the Lie algebra of the stabilizer $G_{x_0}$ of the point $x_0$.
It will be computed in  Proposition~\ref{prop!comput_of_lieGx0ForK5}.

\begin {cor}    \label{cor!aboufFvaForK5}
Assume $v \in V'$  is an invariant.
     Write  $v  =  v_a + v_b$   with  $v_a \in V_a'$ and  $v_b \in V_b'$. 
     Set   $F = Z_a \cap \fg_{x_0}$. Then
                           $F \cdot v_a  \subset Z_a \cdot x_0$.
\end {cor}

\begin {proof}       If $n = 2$ then $v = v_a$ and the result is obvious.   Assume $n \geq 3$
  and let $z \in F$.   Since $F \subset \fg_{x_0}$  and $v$ is an invariant 
    we have  $z \cdot v \in \fg \cdot  x_0$. Hence $z \cdot v_a + z \cdot v_b  \in \fg \cdot x_0$.
          Since $z \in Z_a$,  by Proposition~\ref {prop!aboutZaVasubsetVaForK5}   $z \cdot v_a \in V_a'$ and  $z \cdot v_b \in V_b'$.
          Hence    $z \cdot v_a + z \cdot v_b$    is the decomposition of  $z \cdot v$  with components in  
          $V_a'$ and $V_b'$.
          Using Corollary~\ref {cor!firstcorollary23ForK5} we have that    $z \cdot v_a \in  Z_a \cdot x_0$.  
\end {proof}

Direct computations give the following two propositions.     

\begin {prop}    \label{prop!comput_of_Gx0ForK5}
The stabilizer subgroup  $G_{x_0}$ is equal to the set  of   $(h_1,h_2)  \in \Sp_{2n} \times \GL_2 $
which have the property that there exists  $b \in \CC^* $ and $a_1, a_2 \in \CC$ 
such that   \\   
  $   \phantom{=}  h_1(e_1) = be_1, \quad   h_1(e_2) = b^{-1} e_2 + a_1 e_1,  \quad 
        h_2(g_1) = b^{-1}g_1, \quad   h_2(g_2) = bg_2 + a_2 g_1$.  
\end {prop}

\begin {prop}    \label{prop!comput_of_lieGx0ForK5}
The following set         
 \[
    \{   a_{12}  \}  \; \cup \;  \{ a_{ij} : \;  1 \leq i \leq n, \;  3 \leq j \leq n \} \;  \cup 
     \{  b_{ij} :  3 \leq i \leq j \leq n \}  \; \cup 
 \]
 \[
     \{  c_{ij}:  1 \leq i \leq j \leq n \} \;  \cup \;  
     \{     d_{12} \}   \cup \{  d_{11} - a_{11} - d_{22} + a_{22} \} 
\]
   is a basis for  $\fg_{x_0}$. 
\end {prop}

\begin {lemma}  \label{lem!compatibilityofLieGinvar_forK5}
   Assume $v \in V'$ is an invariant for K5($n$). 
  Write $v = v_a + v_b$ with $v_a \in V_a'$ and $v_b \in V_b'$.
Then     $[v_a] \in (V'/\fg \cdot x_0)^{\fg_{x_0}}$ for K5($2$).
\end {lemma}

\begin {proof} 
Using  Proposition~\ref{prop!comput_of_lieGx0ForK5},  
which implies compatibility of $\fg_{x_0}$ for
K5($n$) as $n$ varies, the result
follows by  Corollary~\ref{cor!aboufFvaForK5}.
\end {proof}

\begin {prop}  \label {prop!about_int_depend_of_weightsForK5}
   Assume $v \in V'$  is a $\rhoprime$-weight vector such that 
   $[v] \in \Vggpr$ and $[v] \not= 0$.  Denote by $\beta$ the
   $\rhoprime$-weight of $v$. Then $\beta$ is a linear combination
   with integer coefficients of the $\lambda_i$. 
\end {prop}

\begin {proof}         Write $v = w_1 + w_2 + w_3$ with $w_j \in V_j'$.
Since $v \not= 0$ there exists $i$ with $1 \leq i \leq 3$  such that $w_i \not= 0$.
Then  the arguments in the proof of \cite[Lemma 2.17(c)]{degensphermod}
also work here, taking into account
that by the definition of $\rhoprime$ 
we have that $w_i$ is a $T$-weight vector with $T$-weight equal to 
$\lambda_i - \beta$.
\end {proof}

\subsection   {Analysis of  the invariants of K5 with a high index}
    
Assume $n \geq 2$ and we are in  K5($n$) case. 
Proposition~\ref{prop!keyproposforhighinvar_ForK5}
will give a strong restriction on the 
invariants $v$ with  $v \notin V_a'$.
Recall $H = \sum_{i=1}^n e_i \wedge f_i$.
We set  $\gamma_2 = \epsilon_1 + \epsilon_2$ and $q^{(2)} = H \otimes g_1 \wedge g_2$.

\begin {remark}   \label{rem!aboutmithatareinvariants_ForK5}
 We have computed $G_{x_0}$  in  
 Proposition~\ref{prop!comput_of_Gx0ForK5}.
 A small computation shows that
$q^{(2)}$  is an  invariant.
\end {remark}

\begin {lemma}  \label{lem!restriction_of_weights_no1_ForK5}
Assume $v \in V'$  such that $0 \not= [v] \in \Vggpr$. Assume
$v$ is a $\rhoprime$-weight vector. 
Denote by $\beta$ the  $\rhoprime$-weight of $v$.
Assume  $w$ is an element of the  monomial basis of $V'$ such that
$w \in V_b'$  and  a nonzero multiple of $w$ is a term of $v$.
Then  $n \geq 3, \beta = \gamma_2,  w \in  V_2'$ and  
there exists $i$ with $3 \leq i \leq n$ such that $w = e_i \wedge f_i \otimes g_1 \wedge g_2$.
\end {lemma}

\begin {proof} 
Denote by $\Gamma$ the $\ZZ$-span of the weights   $\epsilon_1,  \epsilon_2, \epsilon_1', \epsilon_2'$.
We will use that by Proposition~\ref{prop!about_int_depend_of_weightsForK5}  $\beta \in  \Gamma$.
The assumption $w \in V_b'$ implies that there exists $j$ with $3 \leq j \leq 2n-2$
such that $e_j$ appears in $w$.   If $w \in V_1'$ then $w = e_j \otimes g_p$ with
$1 \leq p \leq 2$ which implies that $\beta \notin \Gamma$, a contradiction. 
Hence $w \in V_2'$, so  $w =e_p \wedge e_q \otimes g_1 \wedge g_2$ for
some $1 \leq p < q  \leq n$, or $w =e_p \wedge f_q \otimes g_1 \wedge g_2$
for some $1 \leq p,q \leq n$, or   $w =f_p \wedge f_q \otimes g_1 \wedge g_2$
for some $1 \leq p < q  \leq n$.
The first and the third cases are impossible, since then $\beta \notin \Gamma$.
The second case is possible if and only if $p=q$.
\end {proof}

We need the following lemma, which restricts further 
the candidate invariants.

\begin {lemma}   \label {lem!first_exclusion_forV2_ForK5}
Assume $c_t \in \CC$,  for $ 1 \leq t \leq n$.
Set  $ z  =   (\sum_{t=1}^n c_t e_t \wedge f_t) \otimes   g_1 \wedge g_2$.
Assume $v \in V'$ is an invariant  which is also a 
$\rhoprime$-weight vector with $\rhoprime$-weight $\gamma_2$
such that, for all $1 \leq t \leq n$ and $d \in \CC^*$,  we have that
$de_t \wedge f_t  \otimes  g_1 \wedge g_2$ is not  a term of  $v - z$.
Then   $c_i = c_j$  for all  $3 \leq i < j \leq n$.
\end {lemma}

\begin {proof}   
   Fix  $i,j$  with  $3 \leq i < j \leq n$.   We assume $c_i \not= c_j$
   and we will get a contradiction.      Set $ w  = e_i \wedge f_j  \otimes g_1 \wedge g_2$.
     
    We act by  the element  
    $a_{ij} =   ( e_j \mapsto e_i ,  f_i \mapsto -f_j)$  which, by 
    Proposition~\ref{prop!comput_of_lieGx0ForK5},  is in   $\fg_{x_0}$.
    We have  $a_{ij} \cdot z = (c_i-c_j)w$.
    Lemma~\ref{lem!restriction_of_weights_no1_ForK5}, which restricts the
    terms that can appear in $v$, implies that no nonzero muliple
    of $w$ can appear as a term of $a_{ij} \cdot (v-z)$.  
    Hence   $(c_i-c_j)  w $
    appears  as a term  of  $a_{ij} \cdot v$.  Since $v$ is an invariant,
    we get that  $a_{ij} \cdot v \in \fg \cdot x_0$, hence 
    there exists $z \in \fg$ such that $(c_i-c_j)  w $ is a term of 
    $z \cdot x_0$.  Since $x_0$ has all indices less or equal to $2$ 
    it follows that each term of $z \cdot x_0$ can have at most one index $\geq 3$.
    This is a contradiction, since $w$ has two distinct indices, namely  
    $i,2n+1-j$, greater or equal than $3$.
\end {proof}

The following proposition is an important step for  the
reduction of the problem of 
invariants for the case K5($n$) to the case K5($2$).

\begin {prop}    \label{prop!keyproposforhighinvar_ForK5}
Assume $v \in V'$  such that $0 \not= [v] \in \Vggpr$. Assume
$v$ is a $\rhoprime$-weight vector. 
Denote by $\beta$ the  $\rhoprime$-weight of $v$. Write $v=v_a+v_b$
with $v_a \in V_a'$ and $v_b \in V_b'$.  Assume $v_b \not= 0$.
Then $\beta = \gamma_2$  and   there exists $c \in \CC^*$ 
such that  $v - cq^{(2)}$ is an invariant contained in
$V_a'$.  
\end {prop}

\begin {proof} 
 Lemma~\ref{lem!restriction_of_weights_no1_ForK5}
implies that $\beta = \gamma_2$.  
Combining Lemmas~\ref{lem!restriction_of_weights_no1_ForK5} and
\ref{lem!first_exclusion_forV2_ForK5} 
it follows that there exist $c \in \CC^*$ such that 
$v - c q^{(2)} \in V_a'$.  
By  Remark~\ref{rem!aboutmithatareinvariants_ForK5} $q^{(2)}$ 
is an invariant, hence 
$v - cq^{(2)}$ is an invariant contained in $V_a'$.
\end {proof}

\subsection   {The invariants of K5}

Recall $H_{s} = \sum_{i=1}^2 e_i \wedge f_i$ and $\gamma_2 = \epsilon_1 + \epsilon_2$.
We define the following $\rhoprime$-weights and weight vectors in $V'$:
We set $\gamma_1 =  \epsilon_1- \epsilon_2 + \epsilon_1' - \epsilon_2'$,
and   $r_{1} =  e_2 \otimes  g_2$.  We set  $r_{2,1} = H_{s} \otimes g_1 \wedge g_2$ and
$r_{2,2} =  e_1 \wedge f_1  \otimes g_1 \wedge g_2$. Then $r_{1}$ is a 
$\rhoprime$-weight vector with $\rhoprime$-weight $\gamma_1$ and 
$r_{2,1}$ and $r_{2,2}$ are $\rhoprime$-weight vectors with $\rhoprime$-weight $\gamma_2$.

An easy direct computation gives the following proposition.

\begin {prop}  \label{prop!invariants_forn2_ForK5}
 Assume we are in K5(2).  Then  $\Vggpr  =  (V'/\fg \cdot x_0)^{\fg_{x_0}}$
and the classes in $V'/\fg \cdot x_0$ of the elements 
\[
    r_{1},\; \; r_{2,1}, \; \; r_{2,2}
\] 
is a basis for the vector space $\Vggpr$.
\end {prop}

\begin {prop}  \label{prop!invariants_for_general_n_ForK5}
Assume $n \geq 2$ and we are in K5($n$).
Then the classes in  $V'/\fg \cdot x_0$ of the elements 
\[
    r_{1},\; \;  q^{(2)}, \; \; r_{2,2}
\] 
is a basis for the vector space $\Vggpr$.
\end{prop}

\begin {proof}
  If $n=2$ then the result follows from Proposition~\ref{prop!invariants_forn2_ForK5}.
   So we assume that $n \geq 3$.

 We have computed $G_{x_0}$  in  
  Proposition~\ref{prop!comput_of_Gx0ForK5}.
  A small calculation  shows that the vectors in the statement 
  of the present proposition are indeed invariants for K5($n$).

   We will use that if an invariant $v$ for K5($n$) is an element of $V_a'$ then, 
   by  Lemma~\ref{lem!compatibilityofLieGinvar_forK5},  
   $[v] \in (V'/\fg \cdot x_0)^{\fg_{x_0}} \;$ for K5($2$), 
    hence  by  Proposition~\ref{prop!invariants_forn2_ForK5}
    $v$  is an invariant for K5($2$).

   Assume $v \in V'$ is an invariant which is also a $\rhoprime$-weight 
   vector with $\rhoprime$-weight $\beta$. Write $v = v_a + v_b$ with
   $v_a \in V_a'$ and $v_b \in V_b'$. 
   We will show that $[v]$ is in the linear subspace of
   $V'/\fg \cdot x_0$ spanned by the classes of the elements in
    the statement of the present proposition.  Since these classes
    are linearly independent it will then follow that 
    the classes are a basis for $\Vggpr$.

   We first prove that there exists $i$ with $1 \leq i \leq 2$ such that
   $\beta = \gamma_i$. If $v_b \not= 0$ it follows from 
   Proposition~\ref{prop!keyproposforhighinvar_ForK5} that 
   $\beta = \gamma_2$. If $v_b = 0$ then 
   $v$ is an invariant for K5($2$), hence the
   existence of $i$ such that  $\beta = \gamma_i$ follows from
   Proposition~\ref{prop!invariants_forn2_ForK5}.

   Assume  $\beta = \gamma_1$.
   By Proposition~\ref{prop!keyproposforhighinvar_ForK5}
   $v_b = 0$ and hence $v$ is an invariant
   for the case K5($2$). 
   By Proposition~\ref{prop!invariants_forn2_ForK5}
   the  vector subspace  of $\Vggpr$  consisting of vectors
   with $\rhoprime$-weight  $\gamma_1$ is $1$-dimensional 
   spanned  by  $[r_1]$. Hence $[v]$ is in the linear span of $[r_1]$.

   Assume $\beta = \gamma_2$. 
   It follows from Proposition~\ref{prop!keyproposforhighinvar_ForK5}
   that  there exists $c_1 \in \CC$  such that  $v - c_1q^{(2)}$ is an invariant contained in
   $V_a'$.  Consequently, by Proposition~\ref{prop!invariants_forn2_ForK5} there
   exist $c_2, c_3 \in \CC$ with
   \begin{equation}    \label{eqn!forvc1c3912}
       v - c_1 q^{(2)} - (  c_2 r_{2,1} +  c_3 r_{2,2})  \in  \fg \cdot x_0.
   \end{equation}
 We will show $c_2 = 0$.  Assume $c_2 \not= 0$ and we will get a contradiction.
 We act by the element   $a_{23} =  (e_3 \mapsto e_2, f_2 \mapsto -f_3) \in \fg_{x_0}$.
 Since for all $b \in \fg_{x_0}$ we have $b \cdot (\fg \cdot x_0) \subset (\fg \cdot x_0)$,
 we get $a_{23} \cdot (\fg \cdot x_0) \subset  \fg \cdot x_0$.
 Since $v, q^{(2)}$ are invariants, we have 
 that $a_{23} \cdot (v - c_1q^{(2)}) \in \fg \cdot x_0$.
Hence, Equation~(\ref{eqn!forvc1c3912}) implies 
$ a_{23} \cdot (  c_2 r_{2,1} +  c_3 r_{2,2}) \in \fg \cdot x_0$. 
Since $r_{2,2}$ is an invariant and $c_2 \not= 0$ it follows that 
$a_{23} \cdot r_{2,1} \in  \fg \cdot x_0$, 
hence  $ e_2 \wedge f_3  \otimes  g_1 \wedge g_2   \in \fg \cdot x_0$.    
 This is a contradiction, since by Lemma~\ref {lem:image_of_Vbeta} (for $t=1$) 
   the set of  nonzero elements of $\fg \cdot x_0$ with $\rhoprime$-weight 
   $ \epsilon_2 + \epsilon_3$ is equal to $ \; 
 \{ d (b_{13} \cdot x_0) :       d \in \CC^* \} =
   \{ d ( f_3 \otimes g_1   -  e_2 \wedge f_3  \otimes  g_1 \wedge g_2 ) :
       d \in \CC^* \}.$    
\end {proof}

\begin {prop}  \label{prop!invariants_inV_for_general_n_ForK5}
Assume $n \geq 2$ and we are in K5($n$).
Then the classes in  $V/\fg \cdot x_0$ of the following elements of $V$
\[
     r_1,   \; \; \; \;  n r_{2,2}- q^{(2)}
\]
is a basis for the vector space $\Vgg$.
\end {prop}

\begin{proof}   
     By   \cite[Theorem 17.5]{fulton-harris}
    $V(\omega_2) \subset \wedge ^2  \CC^{2n}$ is the 
     kernel of  the linear map $ \wedge^2 \CC^{2n} \to \CC$
     uniquely specified  by  $  u_1  \wedge  u_2   \mapsto  \Omega (u_1, u_2)$ for all $u_1, u_2 \in \CC^{2n}$.
     Hence a basis of $V(\omega_2)$ is
     \[
         e_i \wedge e_j,   \;   \;   \;   \;   f_i \wedge f_j,   \;   \;   \;   \; e_a \wedge f_b,  \; \; \; \;   e_k \wedge f_k - e_1 \wedge f_1      
     \]
     with indices  $ 1 \leq i < j \leq n$,   $\; 2 \leq k \leq n $,    $ \; 1 \leq a ,b \leq n$    with  $a \not= b$.

    Using Proposition~\ref{prop!invariants_for_general_n_ForK5}, the result follows by an easy computation. 
\end {proof}

\begin {prop}  \label{prop!invariants_inV_that_extend_ForK5}
Assume $n \geq 2$ and we are in K5($n$).
Then the  vector subspace of  $\Vgg$   consisting of the vectors 
with the property that the section in $H^0 (G \cdot x_0, \shN_{X_0})^G$ they induce
extends to $X_0$ is equal to $\Vgg$.  
\end{prop}

\begin{proof}  
 It is enough to show that for each of the
$2$ basis elements
in the statement  of Proposition~\ref{prop!invariants_inV_for_general_n_ForK5}
the induced section  in $H^0 (G \cdot x_0, \shN_{X_0})^G$
extends to $X_0$.   Recall $x_0 = \sum_{i=1}^3 v_i$ with $v_i \in V_i$.
  For $1 \leq i \leq 3$ we set  $ w_i = x_0 - v_i$.  
Using \cite[Proposition~3.1]{degensphermod} we have that, for $i=1,2$,  the codimension in $X_0$ 
of the    $G$-orbit of $w_i$   is $\geq 2$. 

We have
$\gamma_1 = 2 \lambda_1 - \lambda_2$. It follows 
from Corollary~\ref{cor:extending-sections}  that   the equivariant section in
$H^0( G \cdot x_0, \shN_{X_0})^G$  defined by $r_1$  extends to $X_0$.  

We have  $\gamma_2 = \lambda_2 - \lambda_3$.
It follows  from Corollary~\ref{cor:extending-sections}  that   the equivariant section in
$H^0( G \cdot x_0, \shN_{X_0})^G$  defined by $nr_{2,2}-q^{(2)}$  extends to $X_0$.
\end{proof}

\def\cprime{$'$} \def\cprime{$'$} \def\cprime{$'$} \def\cprime{$'$}
  \def\cprime{$'$}
\providecommand{\bysame}{\leavevmode\hbox to3em{\hrulefill}\thinspace}
\providecommand{\MR}{\relax\ifhmode\unskip\space\fi MR }
\providecommand{\MRhref}[2]{%
  \href{http://www.ams.org/mathscinet-getitem?mr=#1}{#2}
}
\providecommand{\href}[2]{#2}


\begin{thebibliography}{vLCL92}

\bibitem[AB05]{alexeev&brion-modaff}
Valery Alexeev and Michel Brion, \emph{Moduli of affine schemes with reductive
  group action}, J. Algebraic Geom. \textbf{14} (2005), no.~1, 83--117.
  \MR{MR2092127 (2006a:14017)}

\bibitem[ACF14]{avdeev&cupit-irrcomps-arxivv2}
Roman Avdeev and St{\'e}phanie Cupit-{F}outou, \emph{On the irreducible
  components of moduli schemes for affine spherical varieties},
  arXiv:1406.1713v2 [math.AG], 2014.

\bibitem[BCF08]{bravi&cupit}
P.~Bravi and S.~Cupit-Foutou, \emph{Equivariant deformations of the affine
  multicone over a flag variety}, Adv. Math. \textbf{217} (2008), no.~6,
  2800--2821. \MR{MR2397467 (2009a:14061)}

\bibitem[Bou68]{bourbaki-geadl47}
N.~Bourbaki, \emph{\'{E}l\'ements de math\'ematique. {F}asc. {XXXIV}. {G}roupes
  et alg\`ebres de {L}ie. {C}hapitre {IV}: {G}roupes de {C}oxeter et syst\`emes
  de {T}its. {C}hapitre {V}: {G}roupes engendr\'es par des r\'eflexions.
  {C}hapitre {VI}: syst\`emes de racines}, Actualit\'es Scientifiques et
  Industrielles, No. 1337, Hermann, Paris, 1968. \MR{MR0240238 (39 \#1590)}

\bibitem[BR96]{benson-ratcliff-mf}
Chal Benson and Gail Ratcliff, \emph{A classification of multiplicity free
  actions}, J. Algebra \textbf{181} (1996), no.~1, 152--186. \MR{MR1382030
  (97c:14046)}

\bibitem[Bri10]{brion-cirmactions}
Michel Brion, \emph{Introduction to actions of algebraic groups}, Les cours du
  CIRM \textbf{1} (2010), no.~1, 1--22.

\bibitem[Bri13]{brion-ihs}
\bysame, \emph{Invariant {H}ilbert schemes}, Handbook of moduli. {V}ol. {I},
  Adv. Lect. Math. (ALM), vol.~24, Int. Press, Somerville, MA, 2013,
  pp.~64--117. \MR{3184162}

\bibitem[BVS15]{msfwm}
Paolo Bravi and Bart Van~Steirteghem, \emph{The moduli scheme of affine
  spherical varieties with a free weight monoid}, Int. Math. Res. Not. IMRN
  (2015), 44 pages, advance access, doi: 10.1093/imrn/rnv281.

\bibitem[CLS11]{cox_little_schenck-toricbook}
David~A. Cox, John~B. Little, and Henry~K. Schenck, \emph{Toric varieties},
  Graduate Studies in Mathematics, vol. 124, American Mathematical Society,
  Providence, RI, 2011. \MR{2810322 (2012g:14094)}

\bibitem[FH91]{fulton-harris}
William Fulton and Joe Harris, \emph{Representation theory}, Graduate Texts in
  Mathematics, vol. 129, Springer-Verlag, New York, 1991, A first course,
  Readings in Mathematics. \MR{1153249 (93a:20069)}

\bibitem[GS]{M2}
Daniel~R. Grayson and Michael~E. Stillman, \emph{Macaulay2, a software system
  for research in algebraic geometry}, available at {\tt
  http://www.math.uiuc.edu/Macaulay2/}.

\bibitem[GW09]{goodman-wallach-gtm}
Roe Goodman and Nolan~R. Wallach, \emph{Symmetry, representations, and
  invariants}, Graduate Texts in Mathematics, vol. 255, Springer, Dordrecht,
  2009. \MR{2522486 (2011a:20119)}

\bibitem[JR09]{jansou-ressayre}
S{\'e}bastien Jansou and Nicolas Ressayre, \emph{Invariant deformations of
  orbit closures in {$\mathfrak{sl}(n)$}}, Represent. Theory \textbf{13}
  (2009), 50--62. \MR{MR2485792}

\bibitem[Kno98]{knop-rmks}
Friedrich Knop, \emph{Some remarks on multiplicity free spaces}, Representation
  theories and algebraic geometry (Montreal, PQ, 1997), NATO Adv. Sci. Inst.
  Ser. C Math. Phys. Sci., vol. 514, Kluwer Acad. Publ., Dordrecht, 1998,
  pp.~301--317. \MR{MR1653036 (99i:20056)}

\bibitem[Lea98]{leahy}
Andrew~S. Leahy, \emph{A classification of multiplicity free representations},
  J. Lie Theory \textbf{8} (1998), no.~2, 367--391. \MR{MR1650378
  (2000g:22024)}

\bibitem[Los09]{losev-knopconj}
Ivan~V. Losev, \emph{Proof of the {K}nop conjecture}, Ann. Inst. Fourier
  (Grenoble) \textbf{59} (2009), no.~3, 1105--1134. \MR{MR2543664}

\bibitem[Lun01]{luna-typeA}
D.~Luna, \emph{Vari\'et\'es sph\'eriques de type {$A$}}, Publ. Math. Inst.
  Hautes \'Etudes Sci. (2001), no.~94, 161--226. \MR{1896179 (2003f:14056)}

\bibitem[PVS12]{degensphermod}
Stavros~Argyrios Papadakis and Bart Van~Steirteghem, \emph{Equivariant
  degenerations of spherical modules for groups of type {$\mathsf{A}$}}, Ann.
  Inst. Fourier (Grenoble) \textbf{62} (2012), no.~5, 1765--1809, extended
  version at arXiv:1008.0911v3 [math.AG]. \MR{3025153}

\bibitem[Tim11]{timashev-embbook}
Dmitry~A. Timashev, \emph{Homogeneous spaces and equivariant embeddings},
  Encyclopaedia of Mathematical Sciences, vol. 138, Springer, Heidelberg, 2011,
  Invariant Theory and Algebraic Transformation Groups, 8. \MR{2797018
  (2012e:14100)}

\bibitem[TY05]{tauvel-yu-laag}
Patrice Tauvel and Rupert W.~T. Yu, \emph{Lie algebras and algebraic groups},
  Springer Monographs in Mathematics, Springer-Verlag, Berlin, 2005.
  \MR{MR2146652 (2006c:17001)}

\bibitem[vLCL92]{LiE1992}
Marc A.~A. van Leeuwen, Arjeh~M. Cohen, and Bert Lisser, \emph{{LiE}, a package
  for {L}ie group computations}, Computer Algebra Nederland, Amsterdam, 1992,
  available at \url{http://www-math.univ-poitiers.fr/~maavl/LiE/}.

\bibitem[VP72]{vin&pop-quasi}
{\`E}.~B. Vinberg and V.~L. Popov, \emph{A certain class of quasihomogeneous
  affine varieties}, Izv. Akad. Nauk SSSR Ser. Mat. \textbf{36} (1972),
  749--764, English translation in Math. USSR Izv. 6 (1972), 743--758.
  \MR{MR0313260 (47 \#1815)}

\end{thebibliography}
\end{document}